\documentclass[11pt,reqno]{amsart}

\usepackage[utf8x]{inputenc}
\usepackage[T1]{fontenc}
\usepackage[english]{babel}
\usepackage{amsmath,amssymb,amsfonts,amsthm}
\usepackage{amsmath,amsthm,amsfonts,amssymb,eucal,hyperref}
\usepackage[pdftex]{graphicx}
\usepackage{color}
\usepackage{enumerate}
\usepackage[super]{nth}
\usepackage{enumitem}
\usepackage{mathtools}
\usepackage{tikz}
\mathtoolsset{showonlyrefs}

\newtheorem{thm}{Theorem}[section]
\newtheorem{lem}[thm]{Lemma}
\newtheorem{prop}[thm]{Proposition}
\newtheorem{cor}[thm]{Corollary}

\theoremstyle{definition}
\newtheorem{defi}[thm]{Definition}
\newtheorem{rem}[thm]{Remark}

\newcommand{\de}{\, \mathrm{d}}
\newcommand{\del}{\partial}

\newcommand{\N}{\mathbb{N}}

\newcommand{\Q}{\mathbb{Q}}
\newcommand{\R}{\mathbb{R}}
\newcommand{\C}{\mathbb{C}}

\newcommand{\D}{\mathbb{D}}

\renewcommand{\S}{\mathbb S}

\newcommand{\CV}{\mathcal{V}}

\newcommand{\abs}[1]{\left\lvert #1 \right\rvert}
\newcommand{\set}[1]{\left\{ #1 \right\}}

\newcommand{\bo}\boldsymbol{}
\newcommand{\bigo}[2][]{O_{#1}\left( #2 \right)}
\newcommand{\smallo}[2][]{o_{#1}\left( #2 \right)}

\DeclareMathOperator{\Op}{Op}

\DeclareMathOperator{\sgn}{sgn}

\DeclareMathOperator{\per}{per}
\DeclareMathOperator{\area}{area}
\DeclareMathOperator{\mult}{mult}

\newcommand{\ceil}[1]{\left\lceil#1\right\rceil}

\newcommand{\pardiff}[2]{\frac{\partial #1}{\partial #2}}

\renewcommand{\hat}{\widehat}
\renewcommand{\tilde}{\widetilde}
\newcommand{\eps}{\varepsilon}
\renewcommand{\phi}{\varphi}

\renewcommand{\mod}[1]{\,({\rm mod}\,#1)}

\DeclareMathOperator{\DN}{DN}

\tikzstyle{grey}=[fill={rgb,255: red,191; green,191; blue,191}, -, draw=black]
\tikzstyle{white}=[-, fill=white]
\tikzstyle{hole}=[-, very thick, fill=white, draw=black, pattern=crosshatch dots]
\tikzstyle{dashes1}=[-, very thick, dashed, fill={rgb,255: red,128; green,128; blue,128}]
\tikzstyle{dashes2}=[-, dotted, fill={rgb,255: red,150; green,150; blue,150}]
\tikzstyle{transparent edge}=[-, draw=none, fill={rgb,255: red,191; green,191; blue,191}]
\usetikzlibrary{patterns}
\tikzset{
  ring shading/.code args={from #1 at #2 to #3 at #4}{
    \def\colin{#1}
    \def\radin{#2}
    \def\colout{#3}
    \def\radout{#4}
    \pgfmathsetmacro{\proportion}{\radin/\radout}
    \pgfmathsetmacro{\outer}{.8818cm}
    \pgfmathsetmacro{\inner}{.8818cm*\proportion}
    \pgfmathsetmacro{\innerlow}{\inner-0.01pt}
    \pgfdeclareradialshading{ring}{\pgfpoint{0cm}{0cm}}%
    {
      color(0pt)=(white);
      color(\innerlow)=(white);
      color(\inner)=(#1);
      color(\outer)=(#3)
    }
    \pgfkeysalso{/tikz/shading=ring}
  },
}

\begin{document}

\title[Dirichlet-to-Neumann spectral invariants]{Spectral invariants of
Dirichlet-to-Neumann operators on surfaces}

\author[J. Lagac\'e]{Jean Lagac\'e}
\address{University College London, Department of Mathematics, 25 Gordon Street, WC1E 6BT, London, UK}
\email{j.lagace@ucl.ac.uk}

\author[S. St-Amant]{Simon St-Amant}
\address{D\'epartement de math\'ematiques et de statistique, Universit\'e de Montr\'eal, CP 6128 succ Centre-Ville, Montr\'eal, QC H3C 3J7, Canada \newline Present address: Department of Pure Mathematics and Mathematical Statistics, Wilberforce Road, Cambridge, CB3 0WB, UK}
\email{sas242@cam.ac.uk}

\begin{abstract}
We obtain a complete asymptotic expansion for the eigenvalues of the
Dirichlet-to-Neumann maps associated with Schr\"odinger operators on compact Riemannian
surfaces with boundary. For the zero potential, we recover the well-known
spectral asymptotics for the Steklov problem. For nonzero potentials, we obtain
new geometric invariants determined by the spectrum of what we call
 the
parametric Steklov problem. In particular, for constant
potentials parametric Steklov problem, the
total geodesic curvature on each connected component of the boundary is a
spectral invariant. Under the constant curvature assumption, this allows us to
obtain some interior information from the spectrum of these boundary operators.
\end{abstract}

\maketitle

\section{Introduction and main result}

\subsection{The Dirichlet-to-Neumann map}

Let $(\Omega,g)$ be a compact Riemannian surface with smooth boundary $\Sigma$ and $\tau
\in C^\infty(\Omega;\R)$. For $\lambda \in \R$, the
Dirichlet-to-Neumann map on $\Omega$
\begin{equation}
\DN_\lambda := \DN_\lambda(\Omega; \tau) : C^\infty(\Sigma) \rightarrow C^\infty(\Sigma)
\end{equation}
is defined as $\DN_\lambda u = \del_\nu \tilde u$, where $\del_\nu$ is the outward pointing normal derivative on $\Sigma$ and $\tilde u$ is the solution to the problem
\begin{equation}
\begin{cases}
(\Delta_g + \lambda\tau) \tilde u = 0 & \text{in } \Omega; \\
\tilde u = u & \text{on } \Sigma.
\end{cases}
\end{equation}

Given $\Omega$ and $\tau$, the map $\DN_\lambda$ is well-defined for all
$\lambda \in \CV \subset \R$, where $\R \setminus \CV$ is a
discrete set consisting in the
Dirichlet eigenvalues of the linear operator pencil $\Delta_g + \lambda \tau$, i.e. the values of $\lambda$ such that the problem
\begin{equation}
\begin{cases}
(\Delta_g + \lambda \tau) u = 0 & \text{in } \Omega; \\
u = 0 & \text{on } \Sigma;
\end{cases}
\end{equation}
admits a non-trivial solution. For fixed $\lambda
\in \CV$, the Dirichlet-to-Neumann map is a self-adjoint elliptic
pseudodifferential operator of order one with principal symbol $\abs{\xi}_g$ \cite{taylorpartial}. Its spectrum is discrete and
accumulating only at infinity,
\begin{equation}
\sigma_0(\Omega; \tau; \lambda) \leq \sigma_1(\Omega; \tau; \lambda)
  \leq \dots \nearrow \infty.
\end{equation}
These eigenvalues are solutions to the eigenvalue problem
\begin{equation}
  \label{prob:ev}
  \begin{cases}
    (\Delta_g + 
    \lambda\tau) u = 0 &\text{in } \Omega; \\
    \del_\nu u = \sigma u & \text{on } \Sigma.
  \end{cases}
\end{equation}

A survey of the general properties of the Steklov problem, i.e. the problem for
$\lambda = 0$, is found in \cite{girouardpolterovich}. When $\lambda = 0$, since
$\tau$ becomes irrelevant, we simply label the eigenvalues $\sigma_j(\Omega,0)$.

\subsection{Spectral asymptotics}

Since $\DN_\lambda$ is a self-adjoint elliptic pseudodifferential operator of
order one with principal symbol $\abs{\xi}_g$, it follows from Weyl's law with sharp remainder (see
\cite{hormanderiv}) that for any $\lambda, \tau$ the eigenvalues
satisfy
\begin{equation}
  \sigma_j(\Omega;\tau;\lambda) = \frac{\pi j}{\per(\Sigma)} + R(j,\lambda,\tau),
\end{equation}
where $\per(\Sigma)$ denotes the length of $\Sigma$ and the remainder
$R(j,\lambda,\tau)$ is a bounded function of $j$ for every fixed $\lambda$ and
$\tau$. Our aim is to obtain a complete asymptotic expansion for
$\sigma_j(\Omega;\tau;\lambda)$. In order to state our results, we generalise
the standard notion of asymptotic expansions.
\begin{defi}
  \label{def:asympequiv}
  Let $\set{a_j}, \set{b_j}$ be two sequences of real numbers. We
  say that they are asymptotically equivalent, and write $a_j \sim b_j$ if for
  all $N \in \N$ there exist $C_N$ and $J_N$ such that for all $j \ge J_N$,
  \begin{equation}
    \abs{a_j - b_j} \le C_N j^{-N}.
  \end{equation}
  More generally, suppose that for every $N \in \N$, $\set{B_j^{(N)}}$ is a
sequence of real numbers. We write
\begin{equation}
a_j \sim B_j^{(\infty)}
\end{equation}
if for every $N \in \N$, there exist $C_N$ and $J_N$ such that for all $j \ge
J_N$,
\begin{equation}
  \abs{a_j - B_j^{(N - 1)}} \le C_N j^{-N}.
\end{equation}
\end{defi}
We observe that the last part of the definition corresponds to the usual notion of
asymptotic expansion
\begin{equation}
  a_j \sim \sum_{n=K}^\infty b_j j^{-n}
\end{equation}
if we take $B_j^{(N)} = \sum_{n=K}^N b_n j^{-n}$ for every $N \in \N$, but our
definition allows for a larger class of asymptotic behaviour.

When
$\Omega$ is simply connected
Rozenblum \cite{rozenblumstek} and Guillemin--Melrose (see \cite{edwards})
proved 
independently that the eigenvalues of $\DN_0$ were asymptotically double and
satisfy the precise asymptotics
\begin{equation}
  \label{eq:edroz}
  \sigma_{2j}(\Omega) \sim \sigma_{2j-1}(\Omega) \sim \frac{2\pi
  j}{\per(\Sigma)} = \sigma_{2j}\left(\frac{\per(\Sigma)}{2\pi} \D\right).
\end{equation}
In other words, $\sigma_{2j}$ and $\sigma_{2j-1}$ have the same complete
asymptotic expansion, and every term except the principal term in that expansion
vanishes. That principal term is given by the Steklov eigenvalues for the disk
with the same perimeter as $\Omega$. Our first result is an extension of this result to the
Dirichlet-to-Neumann operators associated to Schrödinger operators.

\begin{thm}
  \label{thm:sc}
Let $(\Omega, g)$ be a simply connected compact Riemannian surface with smooth boundary
$\Sigma$. For $\lambda \in \R\cap\CV$, the eigenvalues of $\DN_\lambda(\Omega ;
\tau)$ are asymptotically double and admit a complete asymptotic
expansion given by
\begin{equation}
  \label{eq:aexpsc}
  \sigma_{2j} \sim \sigma_{2j-1}  \sim \frac{j}{L} +
  \sum_{n=1}^\infty s_n(\lambda;\Omega) j^{-n},
\end{equation}
where $L = \frac{\per(\Sigma)}{2\pi}$. The coefficients $s_n$ are polynomials in
$\lambda$ of degree at most $n$ with vanishing constant coefficients. They depend on both $\tau$ and the metric in an arbitrarily small
neighbourhood of
$\Sigma$ . If $\tau \equiv 1$, the first two terms are given by
\begin{equation}
  \label{eq:particularparam}
  s_{1}(\lambda;\Omega) = -\frac{\lambda L}{2}, \qquad s_{2}(\lambda;\Omega) =
  \frac{\lambda L}{4\pi} \int_\Sigma k_g \de s
\end{equation}
where $k_g$ is the geodesic curvature on $\Sigma$.
\end{thm}

Just as with the Steklov problem, the righthand side of \eqref{eq:aexpsc} is in
fact an asymptotic expansion for an eigenvalue problem defined on the disk. When $\Omega$ is not simply connected, the
situation is not so simple. Our main theorem shows that in such a case, the
spectrum is asymptotically equivalent to the spectrum of a Dirichlet-to-Neumann
operator associated to a Schr\"odinger operator on a disjoint union of disks. 
When $\lambda = 0$, Girouard, Parnovski, Polterovich and Sher
proved the same result in \cite{GPPS}, whereas Arias-Marco, Dryden, Gordon,
Hassannezhad, Ray and Stanhope proved in \cite{ADGHRS} the equivalent statement
for the eigenvalues of $\DN_0$ on orbisurfaces.

To state the theorem, we introduce notation for a union of non-decreasing sequences,
reordered to be non-decreasing. Let $\Xi = \set{\xi^{(1)},\dotsc, \xi^{(\ell)}}$ be a
finite set of nondecreasing sequences of real numbers accumulating at $\infty$. We denote by $S(\Xi)$ the sequence
$\xi^{(1)} \cup \dotso \cup \xi^{(\ell)}$ rearranged in monotone nondecreasing
order. Here, the union is understood as union of multisets, in other words repeated
elements are kept with their multiplicity.

\begin{thm}
  \label{thm:mc}
  Let $(\Omega,g)$ be a compact Riemannian surface whose smooth boundary
  $\Sigma$ has $\ell$ connected components $\Sigma_1, \dots, \Sigma_\ell$ with
  perimeters $\per(\Sigma_m) = 2 \pi L_m$, $1 \le m \le \ell$. Let $\lambda \in
\R \cap \CV$ and $\tau \in C^\infty(\Omega)$.
  \begin{enumerate}[label=(\Alph*)]
    \item For every $1 \le m \le \ell$, there is a metric $g_m$ on the unit disk,
      $\Upsilon_m$ a collar neighbourhood of $\Sigma_m$, $\tilde \Upsilon_m$
      a collar neighbourhood of $\S^1$ and $\tau_m \in C^\infty(\D)$ such that:
      \begin{enumerate}
        \item There is an isometry $\phi_m : (\tilde \Upsilon_m,g_m) \to
          (\Upsilon_m,g)$.
        \item The restriction of $\tau_m$ to $\tilde \Upsilon_m$ is the pullback
          by $\phi_m$ of $\tau$, in other words $\tau_m\big|_{\tilde \Upsilon_m} =
          \phi_m^* \tau\big|_{\Upsilon_m}$.
        \item      Putting 
      \begin{equation}
        \Omega_\sharp = \bigsqcup_{\ell=1}^m (\D,g_m)
      \end{equation}
    and $\tau_\sharp \in C^\infty(\Omega_\sharp)$ to be equal to $\tau_m$ on each
      components of $\Omega_\sharp$,
      \begin{equation}
        \label{eq:asympequiv}
        \sigma_j(\Omega;\tau;\lambda) \sim \sigma_j(\Omega_\sharp;\tilde
        \tau;\lambda).
      \end{equation}
  \end{enumerate}
\item Any other metric $h$ on $\Omega_\sharp$ and function $\tau' \in
  C^\infty(\Omega_\sharp)$ that satisfies (a) and (b) leads to the same
  asymptotic equivalence \eqref{eq:asympequiv}.
  
\item
  For every $1 \le m \le \ell$ and $N \geq 1$,
  define the sequence $\xi^{(m,N)}$ as $\xi_{0}^{(m,N)}
  = 0$ and for $j \ge 1$,
  \begin{equation}
    \label{eq:sequence}
    \xi_{2j}^{(m,N)} = \xi_{2j-1}^{(m,N)} := \frac{j}{L_m} +
    \sum_{n=1}^N s_n^{(m)}(\lambda;\Omega) j^{-n},
  \end{equation}
  where the coefficients $s_n^{(m)}$
  depend only on $\lambda$, $\tau$ and 
the metric in an arbitrarily small neighbourhood of $\Sigma_m$ in the same way as in
\eqref{eq:aexpsc} (including the case when $\tau \equiv 1$). Let $\Xi^{(N)} = \set{\xi^{(1,N)},\dotsc, \xi^{(\ell,N)}}$. For $\lambda \in \R \cap \CV$, the eigenvalues of $ \DN_\lambda(\Omega;\tau)$ are asymptotically given by
\begin{equation}
  \label{eq:mc}
  \sigma_j \sim S(\Xi^{(\infty)})_j.
\end{equation}
\end{enumerate}
\end{thm}

\begin{figure}[!h]
\centering
\scalebox{0.72}{
\begin{tikzpicture}
		\node  (33) at (-1, 1.75) {};
		\node  (35) at (3.5, 3.75) {};
		\node  (34) at (-1, 5.5) {};
		\node  (32) at (-1, 5.5) {};
		\node  (0) at (-8.25, 6) {};
		\node  (1) at (-8.25, 1.75) {};
		\node  (2) at (-4.25, 5.75) {};
		\node  (3) at (-4.25, 1.75) {};
		\node  (20) at (-5.5, 4.425) {};
		\node  (21) at (-5.5, 3.075) {};
		\node  (22) at (-5.5, 4.25) {};
		\node  (23) at (-5.5, 3.25) {};
		\node  (24) at (-8.25, 6) {};
		\node  (25) at (-8.25, 1.75) {};
		\node  (26) at (-4.25, 5.75) {};
		\node  (27) at (-4.25, 1.75) {};
		\node  (28) at (-10, 3.5) {};
		\node  (29) at (-9.5, 2.4) {};
		\node  (30) at (-10, 3.325) {};
		\node  (31) at (-9.5, 2.575) {};
		\node  (10) at (-0.75, 2.75) {};
		\node  (11) at (1.75, 3) {};
		\node  (8) at (-0.5, 2.75) {};
		\node  (9) at (1.5, 3) {};
		\node  (36) at (-1, 1.75) {};
		\node  (37) at (3.5, 3.75) {};
		\node  (38) at (-1, 5.5) {};
		\node  (43) at (-9.75, 4.575) {};
		\node  (44) at (-7.25, 4.575) {};
		\node  (45) at (-9.05, 4.15) {};
		\node  (46) at (-7.95, 4.15) {};
		\node  (47) at (-0.25, 4.575) {};
		\node  (48) at (2.25, 4.575) {};
		\node  (49) at (0.45, 4.15) {};
		\node  (50) at (1.55, 4.15) {};
		\draw [style=transparent edge] (34.center)
			 to [bend right=15] (33.center)
			 to [in=-90, out=-15] (35.center)
			 to [in=15, out=90] cycle;
		\draw [style=transparent edge] (0.center)
			 to [in=-150, out=165, looseness=2.75] (1.center)
			 to [in=165, out=30, looseness=1.25] (3.center)
			 to [in=-75, out=75, looseness=1.25] (2.center)
			 to [in=-15, out=-165, looseness=1.25] cycle;
		\draw [style=dashes1] (11.center)
			 to [in=90, out=105] (10.center)
			 to [in=-75, out=-90, looseness=1.25] cycle;
		
		\draw [style=white] (9.center)
			 to [in=90, out=105] (8.center)
			 to [in=-75, out=-90, looseness=1.25] cycle;
		\draw [style=hole] (9.center)
			 to [in=90, out=105] (8.center)
			 to [in=-75, out=-90, looseness=1.25] cycle;
		\draw [style=dashes1] (21.center)
			 to [in=165, out=-150, looseness=2.00] (20.center)
			 to [in=30, out=-15, looseness=2.00] cycle;
		\draw [style=white] (23.center)
			 to [in=165, out=-150, looseness=2.25] (22.center)
			 to [in=30, out=-15, looseness=2.00] cycle;
		\draw [style=hole] (23.center)
			 to [in=165, out=-150, looseness=2.25] (22.center)
			 to [in=30, out=-15, looseness=2.00] cycle;
		\draw [style=transparent edge, in=135, out=45, loop] (1.center) to ();
		\draw [in=-150, out=165, looseness=2.75] (24.center) to (25.center);
		\draw [in=165, out=30, looseness=1.25] (25.center) to (27.center);
		\draw [in=-15, out=-165, looseness=1.25] (26.center) to (24.center);
		\draw [style=transparent edge, in=135, out=45, loop] (25.center) to ();
		\draw [style=dashes1] (29.center)
			 to [in=180, out=-165, looseness=2.25] (28.center)
			 to [in=15, out=0, looseness=4.25] cycle;
		\draw [style=white] (31.center)
			 to [in=180, out=-165, looseness=2.50] (30.center)
			 to [in=15, out=0, looseness=5.25] cycle;
		\draw [style=hole] (31.center)
			 to [in=180, out=-165, looseness=2.50] (30.center)
			 to [in=15, out=0, looseness=5.25] cycle;
		\draw [style=transparent edge, bend right=15] (38.center) to (36.center);
		\draw [in=-90, out=-15] (36.center) to (37.center);
		\draw [in=15, out=90] (37.center) to (38.center);
		\draw [dashed, in=-165, out=15] (26.center) to (38.center);
		\draw [dashed, in=165, out=-15] (27.center) to (36.center);
		\draw [bend right=15, looseness=0.75] (43.center) to (45.center);
		\draw [style=white] (46.center)
			 to [bend left=300] (45.center)
			 to [in=-165, out=-15] cycle;
		\draw [bend right=15, looseness=0.75] (46.center) to (44.center);
		\draw [bend right=15, looseness=0.75] (47.center) to (49.center);
		\draw [style=white] (50.center)
			 to [bend left=300] (49.center)
			 to [in=-165, out=-15] cycle;
		\draw [bend right=15, looseness=0.75] (50.center) to (48.center);

		\draw[very thick,fill={rgb,255: red,128; green,128; blue,128}] (-9.5,-1) circle (0.8);
		\draw[very thick,dashed,fill={rgb,255: red,200; green,200; blue,200}] (-9.5,-1) circle (0.64);
		
		\draw[very thick,fill={rgb,255: red,128; green,128; blue,128}] (-5.5,-1) circle (0.65);
		\draw[very thick,dashed,fill={rgb,255: red,200; green,200; blue,200}] (-5.5,-1) circle (0.52);
		
		\draw[very thick,fill={rgb,255: red,128; green,128; blue,128}] (0.5,-1) circle (1);
		\draw[very thick,dashed,fill={rgb,255: red,200; green,200; blue,200}] (0.5,-1) circle (0.8);
		
		\draw [thick,<->] (-9.5,1) to (-9.5,0.2);
		\draw [thick,<->] (-5.5,1) to (-5.5,0.2);
		\draw [thick,<->] (0.5,1) to (0.5,0.2);
		
		\node[scale=1] at (-2.55,3.6) {$\cdots$};
		\node[scale=1] at (-2.55,-1) {$\cdots$};
		\node[scale=1] at (-2.55,0.6) {$\cdots$};
		\node[scale=1,anchor = north west] at (-9.5,0.8) {$\varphi_1$};
		\node[scale=1,anchor = north west] at (-5.5,0.8) {$\varphi_2$};
		\node[scale=1,anchor = north west] at (0.5,0.8) {$\varphi_\ell$};
		\node[scale=1] at (-9.5,-1) {$\D_1$};
		\node[scale=1] at (-5.5,-1) {$\D_2$};
		\node[scale=1] at (0.5,-1) {$\D_\ell$};
		
		\draw (-10.1,2.6) -- (-10.3,2.3);
		\node[scale=1] at (-10.3,2.1) {$\Sigma_1$};
		\draw (-5.8,3.18) -- (-5.85,2.8);
		\node[scale=1] at (-5.8,2.6) {$\Sigma_2$};
		\draw (-0.28,3.14) -- (-0.55,3.5);
		\node[scale=1] at (-0.6,3.7) {$\Sigma_\ell$};
		
		\node[scale=1] at (-7.5,5) {$\Omega$};
		
\end{tikzpicture}}
\caption{Construction of collar neighbourhoods in statement (A) of Theorem \ref{thm:mc}.}
\end{figure}
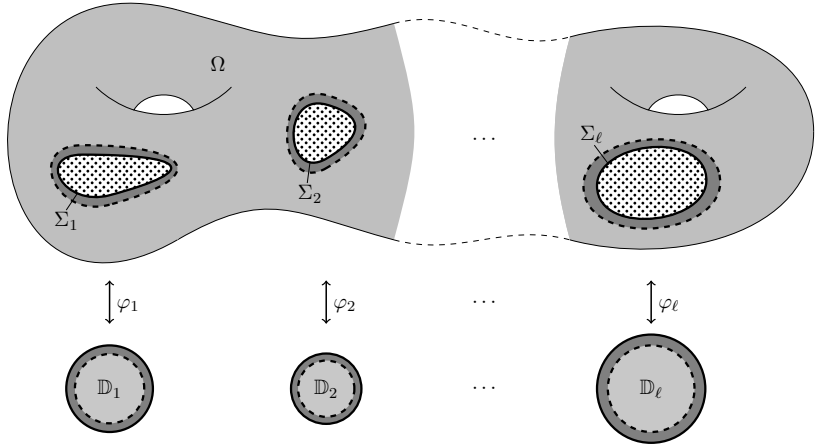

\begin{rem}
  In the previous theorem, statements (B) and (C) are direct consequence of
  statement (A) and Theorem \ref{thm:sc}, along with the observation that the
  spectrum of $\DN_\lambda$ on a disjoint union of surfaces is equal to the
  union of the spectra, repeating multiplicity. As such, we will only need to
  prove statement (A). Nevertheless, statements (B) and (C) are useful to have
  explicitly for our applications to inverse problems.
\end{rem}

Note that Theorem \ref{thm:mc} obviously implies Theorem \ref{thm:sc}. However, the statement
for simply connected surfaces is cleaner and obtained as an intermediate step in proving Theorem \ref{thm:mc}. Hence we state them separately.

\subsection{Inverse spectral geometry}

Inverse problems consist in recovering data of some PDE --- the domain of definition $\Omega$, the metric, the potential,
etc. --- from properties of the operator alone, and inverse spectral geometry
consists in recovering that data from the spectrum only. One of the seminal
questions in that field was asked for the Dirichlet Laplacian by Mark Kac in \cite{kac} and answered
negatively by Gordon, Webb and Wolpert in \cite{GWW}: ``Can one hear the shape
of a drum?'' For this reason, we often say that any geometric data that one can
recover from the spectrum of an operator can be ``heard''.

It is long known and follows from Weyl's law that the total boundary length can
be heard from $\DN_\lambda$. It also follows from the standard theory of the wave trace
asymptotics as developed by Duistermaat and Guillemin \cite{DG} that the
length spectrum --- that is the length of the closed geodesics --- of the
boundary $\Sigma$ can be heard as well. For $\DN_0$, it is shown in
\cite{GPPS} that we can recover the number of connected components, as well as
their lengths. It is also shown that from polynomial eigenvalue asymptotics alone in
dimension two nothing more can be recovered. This can be seen as a consequence of Theorems \ref{thm:sc} and \ref{thm:mc} since the coefficients $s_n$ and $s_n^{(m)}$ are all
polynomials in $\lambda$ that vanish when $\lambda = 0$.

For $\DN_0$, to extract more information different authors have turned to
spectral quantities that have a more global nature. In \cite{PS}, Polterovich
and Sher obtain an asymptotic expansion as $t \to 0$ for the heat trace of
$\DN_0$. From the coefficients, they obtain that the total mean curvature is a
spectral invariant for $d \ge 3$. See also the work of Liu \cite{liu} for
further improvements. In the
case of $\DN_\lambda(\Omega;\tau)$, heat trace asymptotics as well as invariants
deduced from them have also been obtained by Wang and Wang in \cite{WW}, again in
dimension $d \ge 3$. We also refer to the works of Jollivet and Sharafutdinov
\cite{JolSharaf,JolSharaf2018} where they find invariants for simply connected
domains from the zeta function associated with $\DN_0$.

Our main theorem shows that for non-zero potential, one can hear more
information from polynomial eigenvalue asymptotics.

The spectral inverse problem for the Dirichlet-to-Neumann map consists in
extracting information about $\Omega, g, \tau$ and $\lambda$ (or a subset of
these parameters) from the eigenvalues $\{\sigma_j : j \in \N\}$. As an
application of our methods, we will find spectral invariants when $\tau \equiv
1$, and show that we can recover $\lambda$ as well as geometric data on $\Omega$.
For $\lambda = 0$, the problem
has been studied already and is referred to as the Steklov problem. Lee and
Uhlmann have shown in \cite{leeuhlmann} that the map $\DN_0$ (but not
necessarily its spectrum) determines the Taylor series for $g$ close to the
boundary. Girouard, Parnovski, Polterovich and Sher show in \cite{GPPS} that
from polynomial order spectral asymptotics, one can determine the number of
boundary components and each of their lengths, but nothing more.  Our goal is to
obtain more information from the spectrum when $\lambda \neq 0$.

\begin{thm}
  \label{thm:invariance}
  For any $\lambda \in (\R\cap \CV) \setminus\set{0}$, the spectrum of $\DN_\lambda$
  determines the following quantities:
  \begin{itemize}
    \item the number of connected components of the boundary, and their
      respective perimeters;
    \item each coefficient $s_n^{(m)}$ in \eqref{eq:sequence}, and in particular if $\tau \equiv 1$:
    \begin{itemize}
    \item the parameter $\lambda$;
    \item the total geodesic curvature on each boundary component.
    \end{itemize}
  \end{itemize}
\end{thm}

The previous theorem along with the Gauss-Bonnet theorem also yield the
following theorem. We denote by $K$ the Gaussian curvature on $\Omega$.

\begin{thm}
  \label{cor:genusetc}
Let $\Omega$ be a compact orientable Riemannian surface with smooth boundary
$\Sigma$ and genus $\gamma$, and let $\lambda \in (\R \cap \CV) \setminus \set
0$. Then, the quantity
\begin{equation}
4 \pi \gamma +   \int_\Omega K_g \de A_g
\end{equation}
is a spectral invariant of $\DN_\lambda(\Omega; 1)$. In particular, if the
Gaussian curvature is assumed to be a constant $K(\Omega)$, the quantity
\begin{equation}\label{eq:invariant}
4\pi \gamma + K(\Omega)\area(\Omega)
\end{equation}
is a determined by the spectrum of $\DN_\lambda(\Omega; 1)$.
\end{thm}

Here the genus $\gamma$ of $\Omega$ corresponds to the minimal genus of a closed
surface in which $\Omega$ can be topologically embedded. Equivalently, it is the
genus of the closed surface obtained from $\Omega$ by gluing topological disk
onto each boundary component. By restricting the choice of $\Omega$, we can gain
more interior geometric information from the spectrum. Note that while the
Steklov spectrum is not known to determine interior information in general, for
planar domains it is already known from the work of Edward \cite[Theorem 4]{edward2} that
we can get lower bounds for the area.

\begin{cor}
  \label{cor:sphere}
If $\Omega$ is a domain of the standard sphere $\S^2$,
its area is a spectral invariant of $\DN_\lambda(\Omega;1)$ for all $\lambda \in
(\R \cap \CV) \setminus \set 0$.
\end{cor}

\begin{proof}
If $\Omega \subset \S^2$, then $\gamma = 0$ and $K(\Omega) = 1$. This leaves only $\area(\Omega)$ in \eqref{eq:invariant}.
\end{proof}

\begin{cor}\label{cor:flat}
If $\Omega$ is a domain in a flat orientable space form, its genus is a spectral
invariant of $\DN_\lambda(\Omega;1)$ for all $\lambda \in (\R \cap \CV)
\setminus \set 0$.
\end{cor}

\begin{proof}
If $\Omega$ is a domain in a flat orientable space form, then $K(\Omega) = 0$ and only $4\pi\gamma$ remains in \eqref{eq:invariant}.
\end{proof}

The inverse problem for $\DN_\lambda(\Omega;\tau)$ has a concrete interpretation
in terms of the inverse scattering problem. In this context, $\Omega \subset
\R^2$ has anisotropic refraction index $\tau$. Non-destructive testing is the
process of using the far-field data to measure the scattering of an incoming
wave at frequency $\sqrt{\lambda}$ by the obstacle $\Omega$. The inverse
scattering problem consists in recovering then the refraction index $\tau$, as well as the geometry
of $\Omega$. In \cite{CCMM}, it is shown that the far-field data determines the
spectrum of $\DN_\lambda(\Omega;\tau)$, so that any spectral invariant of
$\DN_\lambda$ can be obtained from the far-field data. We have explicit
expressions for geometric quantities related to the boundary of $\Omega$ when
the refraction index is isotropic, i.e. constant. When it is not, we do not give
an explicit value of the coefficients $s_n$, however the algorithmic procedure
to compute them in Sections \ref{sec:trans} and \ref{sec:invariants} applies.
Similarly, Theorem \ref{thm:rhotau} is also valid in that context, giving an
exact expression for the first few invariant quantities. Note that the
coefficients $s_n$ are polynomials of order at most $n$ in $\lambda$ with
vanishing constant coefficient. This means that it is possible to decouple the
coefficients of this polynomial by knowing the asymptotics for
$\lambda_1,\dotsc, \lambda_n$. Physically, this simply means measuring the
scattered far-field data for incoming waves at $n$ different frequencies.

\subsection{Sketch of the proof of Theorem \ref{thm:sc}}

Let us introduce a slightly more general version of Problem \eqref{prob:ev}. For
$\rho : \Sigma \to \R_+$ a strictly positive smooth function, we
consider the eigenvalue problem
\begin{equation}
  \label{prob:gen}
  \begin{cases}
    (\Delta_g + 
    \lambda\tau) u = 0 &\text{in } \Omega; \\
    \del_\nu u = \sigma \rho u & \text{on } \Sigma.
  \end{cases}
\end{equation}
Our first step will be to show that we can reduce Theorems \ref{thm:sc} and
\ref{thm:mc} for Problem \eqref{prob:gen} to proving them for
\begin{equation}\label{prob:disk}
\begin{cases}
-\Delta u = \lambda \tau u & \text{in } \D; \\
\del_\nu u = \sigma \rho u & \text{on } \S^1.
\end{cases}
\end{equation}
In other words, by introducing this extra parameter $\rho$ they only need to be
proved in the case where $\Omega$ is a disk, and $g$ is the flat metric $g_0$ . 

This reduction will be done by following the strategy set out in \cite{GPPS}, where they
glue a disk to a collar neighbourhood of every boundary component, and
discard the rest of the surface. Since the symbol of $\DN_\lambda$ depends
solely on data obtained from a neighbourhood of the boundary, this doesn't change
the symbol of the Dirichlet-to-Neumann map. Mapping these topological disks
conformally to the unit disk in $\R^2$ will multiply the factors $\tau$ and
$\rho$ by a conformal factor, in other words it doesn't change the structure of
the problem.

We then follow the general theory set out by Rozenblum in
\cite{rozenblumalmostsim} to obtain a complete asymptotic expansion of the
eigenvalues of a pseudodifferential operator on a circle in terms of integrals
of its symbol. Note that in \cite{rozenblumalmostsim}, an abstract algorithm is
given to do so, but as is often the case with pseudodifferential symbolic
calculus the expressions become unwieldy very quickly, and the difficulty
resides in extracting actual geometric information out of it. The symbol is easy to compute for $\rho = 1$, $\lambda = 0$,
where it is simply $\abs{\xi}$, with no lower order terms. However, when
$\lambda \neq 0$, this is no longer the case, and it will lead to the
full asymptotic expansion that we obtain.

We obtain the following theorem for the disk.
\begin{thm}\label{thm:disk}
The eigenvalues of Problem \eqref{prob:disk} satisfy the asymptotic expansion
\begin{equation}\label{eq:aexp}
\sigma_{2j} = \sigma_{2j-1} + \bigo{j^{-\infty}} \sim \frac{j}{\int_{\S^1} \rho
\de x} + \sum_{n=1}^{\infty} b_n j^{-n}.
\end{equation}
where the coefficients $b_n$ depend only on $\rho$, $\lambda$ and the values of $\tau$ in a neighbourhood of $\S^1$, as well as their derivatives.
\end{thm}

We will then specialise the previous theorem to the values of $\tau$ and
$\rho$ coming from the conformal mapping between the disk and $\Omega$. We
obtain explicit values of the coefficients $b_n$ in that situation.

\subsection{Plan of the paper} 
In Section \ref{sec:LU}, we make clear our reduction to the disk and compute the full symbol of the Dirichlet-to-Neumann map. In Section \ref{sec:trans} and Section \ref{sec:asymp}, using the method laid out in \cite{rozenblumalmostsim}, we transform the symbol of a general Dirichlet-to-Neumann map on a circle to extract the asymptotic expansion of its eigenvalues. In Section \ref{sec:Steklovasymp}, we specify our results to the case of the parametric Steklov problem in order to show Theorem \ref{thm:sc}. Finally, in Section \ref{sec:invariants}, we prove Theorem \ref{thm:invariance}. There, we use Diophantine approximation to decouple the sequences obtained in Theorem \ref{thm:mc} recursively.

\subsection{Asymptotic notation}

In addition to the asymptotic equivalence introduced in Definition
\ref{def:asympequiv}, we also make use of Landau's asymptotic notation.

\begin{itemize}
  \item For two sequences $\set{a_j}, \set{b_j}$, we write $a_j = \bigo{b_j}$ if there
exist a constant $C > 0$ and $N \in \N$ such that for all $j \ge N$, $\abs{a_j}
\le C b_j$. We note that if the sequence $\set{b_j}$ is strictly positive, this implies the existence of a (potentially larger)
constant $C'$ such that $\abs{a_j} \le C' b_j$ for all $j \in \N$.
\item For two functions $f,g : \R^d \to \R$, we say that $f= \bigo{g}$ if there
  exist $C > 0$ and $R > 0$ such that $\abs{f(x)} \le C g(x)$ for all $\abs{x}
  \ge R$.
\item We write $a_j = \bigo{j^{-\infty}}$ to indicate that for all $N \in \N$, $a_j =\bigo{j^{-N}}$.
  \end{itemize}
  We can observe that the asymptotic equivalence $a_j \sim b_j$ is equivalent to
  $a_j - b_j = \bigo{j^{-\infty}}$.

\subsection*{Acknowledgements}

The research of J.L. was supported by EPSRC grant EP/P024793/1 and the NSERC
Postdoctoral Fellowship. The research of S.St-A. was supported by NSERC's CGS-M and FRQNT's M.Sc. scholarship (B1X). This work is part of his M.Sc. studies at the Universit\'e de
Montr\'eal, under the supervision of Iosif Polterovich. Both authors would like
to thank him for many useful discussions, and relevant comments when this paper
was in its final stages. The authors would also like to thank Alexandre
Girouard, Michael Levitin,
Leonid Parnovski, and Grigori Rozenblum, who read the preliminary version of
this paper and gave useful suggestions. We also thank the two anonymous referees
who made comments improving the clarity of exposition.

\section{The symbol of the Dirichlet-to-Neumann map on surfaces}\label{sec:LU}

This section will be split into two parts : first, we follow Melrose's
factorisation method, as described in \cite{leeuhlmann}. We will see that the symbol
of $\frac{1}{\rho}\DN_\lambda(\Omega;\tau)$ depends only on $\lambda$, $\rho$,
as well as on the restriction of $\tau$ and the metric $g$ in a
neighbourhood of the boundary $\Sigma$. This will allow us to show that we can
reduce the problem at hand to the situation where $\Omega$ is the unit disk
$\mathbb D$. In the second part of this section, we explicitly compute the value
of the symbol for the disk.

 We recall the definition of the rolling radius (see \cite{howard}) and the construction
of Fermi coordinates. Given $x \in \Sigma$, the exponential map defines a normal
geodesic curve $\gamma_{x} : \R_+ \to \Omega$,
\begin{equation}
  \gamma_x(t) = \exp_x(t\nu)
\end{equation}
The cut point of $x$ is the point $\operatorname{cut}_\Sigma(x) =
\exp_x(t_0\nu)$, where $t_0$ is the minimal $t$ such that $\exp(t\nu)$ does not
minimise the distance to $\Sigma$. Smoothness and compactness of the boundary
ensures that such a $t_0 > 0$ exists. The cut locus of $\Sigma$ is the set
\begin{equation}
  \operatorname{cut}(\Sigma) := \set{\operatorname{cut}_\Sigma(x) : x \in
  \Sigma}.
\end{equation}
The rolling radius of $\Omega$ is defined as
$$\operatorname{Roll}(\Omega) := \operatorname{dist}(\Sigma,\operatorname{cut}(\Sigma)),$$
compactness of $\Sigma$ ensures that $\operatorname{Roll}(\Omega) > 0$. It is
called the rolling radius because any open ball of radius $r \le
\operatorname{Roll}(\Omega)$ can roll along $\Sigma$ while always remaining a
subset of $\Omega$. For some $0 < \eps <
\operatorname{Roll}(\Omega)$, let $\Upsilon$ be a collar $\eps$-neighbourhood of
the boundary,
\begin{equation}
\Upsilon := \set{x' \in \Omega : \operatorname{dist}(x', \Sigma) < \eps}.
\end{equation}
Since $\eps < \operatorname{Roll}(\Omega)$, for every
$x' \in \Upsilon$, there is a unique $x \in \Sigma$ and $t < \eps$ such that $x'
= \exp_x(t\nu)$;
set $x' = (x,t)$. The boundary $\Sigma$ is characterised by $\set{t = 0}$,
and the outward normal derivative is given by $\del_\nu = -\del_t$. In these
coordinates, the metric has a much simpler form as
\begin{equation}
g(x') = \tilde g(x') (\mathrm d x)^2 + (\mathrm d t)^2,
\end{equation}
for some positive function $\tilde g$. The Laplacian reads
\begin{equation}
-\Delta_g = D_t^2 -  \frac i 2 (\del_t \log \tilde g) D_t + \tilde g^{-1} D_x^2
- \frac{i}{2} \frac{(\del_x \tilde g) }{\tilde g^2} D_x,
\end{equation}
where $D_x = - i \del_x$ and $x$ now denotes an arc length parameter along $\Sigma$.

\subsection{Reduction to the disk}

We start by observing that \cite[Propositions 1.1 and 1.2]{leeuhlmann} applies to the Schrödinger operator $H = -\Delta - \lambda\tau$.
\begin{prop} \label{prop:facto}
There is a family $A(x,t,D_x)$ of pseudodifferential operators depending smoothly on $t$ such that
\begin{equation}
- \Delta_g - \lambda \tau(x) \equiv (D_t + iE - i A(x,t,D_x))(D_t + i A(x,t,D_x)) \quad \mod{\Psi^{-\infty}},
\end{equation}
where
\begin{equation}
E := - \frac{i}{2}(\del_t \log \tilde g).
\end{equation}
\end{prop}

\begin{proof}
The proof follows that of \cite[Proposition 1.1]{leeuhlmann} in computing the
symbol of $A$ recursively. Their construction only relies on ellipticity of $H$,
and the fact that the only derivatives in $t$ are in $\Delta_g$.   
\end{proof}
\begin{rem}
  In subsection \ref{sec:symboldisk}, we make this recursive computation of the
  symbol explicit for the disk, as we need to obtain concrete values of the
  coefficients in that case. The reader interested in a more detailed proof of
  Proposition \ref{prop:facto} can see that this recursive computation also
  works for a general $\Omega$.
\end{rem}

Proposition \ref{prop:facto} admits the same corollary as in \cite{leeuhlmann}.
\begin{cor} \label{cor:bdry}
Let $r(x, \xi)$ be the symbol of $\frac{1}{\rho}\DN_\lambda(\Omega;\tau)$ and $a(x,t,\xi)$ be  the symbol of $A$. Then 
\begin{equation}
r(x,\xi) = - \frac{a(x,0,\xi)}{\rho(x)}.
\end{equation}
In other words,
\begin{equation}
\frac{1}{\rho} \DN_\lambda(\Omega;\tau) \equiv \frac{-1}{\rho} A\big|_{\Sigma}
\quad \mod{\Psi^{-\infty}}.
\end{equation}
In particular, the symbol of $\frac{1}{\rho}\DN_\lambda(\Omega;\tau)$ depends only on $\lambda, \rho$ and the boundary values of $g,\tau$ and of their derivatives.
\end{cor}
We denote by $\sigma_j(\Omega; \tau; \rho; \lambda)$ the $j^{\mathrm{th}}$ eigenvalue of $\frac{1}{\rho} \DN_\lambda(\Omega,\tau)$.

\begin{lem} \label{lem:isombdry}
Let $\Omega_1$, $\Omega_2$ be compact Riemannian surfaces with smooth boundary
$\Sigma_1, \Sigma_2$. Suppose there exists an isometry $\phi$ between collar
neighbourhoods $\Upsilon_1$ of $\Sigma_1$ and $\Upsilon_2$ of $\Sigma_2$. Let
$\tau \in C^\infty(\Omega_2)$ and $\rho \in C^\infty(\Sigma_2)$. Then,
\begin{equation}
\sigma_j(\Omega_1;\varphi^*\tau;\varphi^*\rho;\lambda) \sim
\sigma_j(\Omega_2;\tau;\rho;\lambda)
\end{equation}
where $\varphi^*$ denotes the pullback by $\varphi$.
\end{lem}

\begin{proof}
By Corollary \ref{cor:bdry}, the operators
$\frac{1}{\varphi^*\rho}\DN_\lambda(\Omega_1;\varphi^*\tau)$ and
$\frac{1}{\rho}\DN_\lambda(\Omega_2;\tau)$ have the same symbol, or in other
words they are equivalent up to smoothing operators:
\begin{equation}
\frac{1}{\varphi^*\rho}\DN_\lambda(\Omega_1;\varphi^*\tau) \equiv \frac{1}{\rho}\DN_\lambda(\Omega_2;\tau) \quad \mod{\Psi^{-\infty}}.
\end{equation}
In \cite[Lemma 2.1]{GPPS}, it is shown that whenever two elliptic selfadjoint
pseudodifferential operators are equivalent up to smoothing operators, then
their eigenvalues are asymptotically equivalent.
\end{proof}

\begin{lem} \label{lem:todisksc}
Let $\Omega$ be a compact simply connected surface with smooth boundary $\Sigma$. Let $\varphi : \overline{\D} \to \Omega$ be conformal. Then, the Steklov problem \eqref{prob:gen} on $\Omega$ is isospectral to the problem
\begin{equation}
\begin{cases}
- \Delta u = \lambda e^{2f} (\varphi^*\tau) u &\text{in } \D; \\
\del_\nu u = \sigma e^f (\varphi^*\rho) u & \text{on } \S^1;
\end{cases}
\end{equation}
where $f : \overline{\D} \to \R$ is such that $\varphi^*g = e^{2f} g_0$.
\end{lem}
\begin{proof}
  It follows directly from the observation, see \cite{JolSharaf}, that the Laplacian and normal derivatives transform under a conformal mapping $\varphi : (\overline{\D},g_0) \to (\Omega;g)$ as
\begin{equation}
\Delta_{g_0} (\varphi^* u) = e^{2f} \varphi^*(\Delta_g u)
\end{equation}
and
\begin{equation}
\del_{\nu,g_0} (\varphi^* u) = e^{f} \varphi^*(\del_{\nu,g} u)
\end{equation}
respectively.
\end{proof}
\begin{lem}
  \label{lem:cutting}
  Let $(\Omega,g)$ be a compact Riemannian surface whose smooth boundary
  $\Sigma$ has $\ell$ connected components $\Sigma_1,\dotsc,\Sigma_\ell$. For
  every $1 \le m \le \ell$, there exist a metric $g_m$ on the unit disk,
  a collar neighbourhood $\Upsilon_m$ of $\Sigma_m$, and
  a collar neighbourhood $\tilde \Upsilon_m$ of $\S^1$ such that $g_m\big|_{\tilde \Upsilon_m}$ is
  isometric to $g\big|_{\Upsilon_m}$.
\end{lem}

\begin{proof}
  Since $\Sigma_m$ is a closed curve, we can assume without loss of generality
  that it is parameterised as $f : [0,2\pi] \to \Sigma_m$. Let $0 < \eps <
  \min\set{\operatorname{Roll}(\Omega),1/2}$ and $\Upsilon_m(\eps)$ be an
  $\eps$-collar neighbourhood of $\Sigma_m$. Let $\tilde \Upsilon_m(\eps) \subset \D$ be defined
  as
  \begin{equation}
    \tilde \Upsilon_m(\eps) := \set{(r, \theta) \in \D : r > 1- \eps}.
  \end{equation}
  Endow $\Upsilon_m$ with Fermi coordinates and define $\phi_m : \tilde
  \Upsilon_m(\eps) \to \Upsilon_m(\eps)$ by $\phi(1-t,\theta) =
(f(\theta),t)$, and let $\tilde g_m = \phi_m^* g$ so that $\tilde g_m$ is
  isometric to $g$. Let
  \begin{equation}
    U := \set{(r,\theta) \in \D : r < 1- \frac \eps 2}
  \end{equation}
  and let $\set{\psi_{\Upsilon},\psi_U}$ be a partition of unity subordinated to
  $(\Upsilon_m(\eps),U)$. Define the metric $g$ on $\D$ as
  \begin{equation}
    g_m = \psi_\Upsilon \tilde g_m + \psi_U g_0,
  \end{equation}
  where $g_0$ is the flat metric. Then, $g_m$ is the desired metric with
  $\Upsilon_m = \Upsilon_m(\eps/2)$ and $\tilde \Upsilon_m = \tilde
  \Upsilon_m(\eps/2)$. 
\end{proof}
\begin{figure}[!h]
\centering
\scalebox{0.8}{
\begin{tikzpicture}
		\node  (10) at (-11.25, 3.25) {};
		\node  (11) at (-6, 4) {};
		\node  (12) at (-5.75, -0.25) {};
		\node  (13) at (-11.25, 0.25) {};
		\node  (4) at (-9.25, 1.5) {};
		\node  (5) at (-6.5, 1.5) {};
		\node  (6) at (-9.625, 1.5) {};
		\node  (7) at (-6.125, 1.5) {};
		\node  (8) at (-10, 1.5) {};
		\node  (9) at (-5.75, 1.5) {};
		\node  (0) at (-11.25, 3.25) {};
		\node  (1) at (-6, 4) {};
		\node  (2) at (-5.75, -0.25) {};
		\node  (3) at (-11.25, 0.25) {};
		\draw [style=transparent edge] (11.center)
			 to [in=30, out=135, looseness=0.75] (10.center)
			 to [bend right=15] (13.center)
			 to [in=-120, out=-45] (12.center)
			 to [in=-45, out=60] cycle;
		\draw [very thick,in=135, out=30, looseness=0.75] (0.center) to (1.center);
		\draw [very thick, in=60, out=-45] (1.center) to (2.center);
		\draw [very thick, in=-120, out=-45] (3.center) to (2.center);
		\draw[dashes2,very thick] (9.center)
			 to [in=-105, out=-75] (8.center)
			 to [in=105, out=75] cycle;
		\draw[-, very thick, dashed, fill={rgb,255: red,100; green,100; blue,100}] (7.center)
			 to [in=75, out=105] (6.center)
			 to [in=-75, out=-105] cycle;
		\draw[white,very thick] (5.center)
			 to [in=-105, out=-75] (4.center)
			 to [in=105, out=75] cycle;
		\draw[hole,very thick] (5.center)
			 to [in=-105, out=-75] (4.center)
			 to [in=105, out=75] cycle;
		
		\draw[very thick,fill ={rgb,255: red,100; green,100; blue,100}] (1,1.5) circle(2);
		\shade[draw=none,ring shading={from {rgb,255: red,200; green,200; blue,200} at 1 to {rgb,255: red,100; green,100; blue,100} at 1.5}]  (1,1.5) circle (1) circle (1.5);
		\draw[very thick,dashed] (1,1.5) circle(1.5);
		\draw[very thick,grey,dotted] (1,1.5) circle(1.05);
			 
		\draw [very thick, bend left=25, looseness=1, <->] (-4,1.5) to (-2,1.5);
		\node [scale = 1.5, anchor = south] at (-3,1.8) {$\varphi_m$};
		\node [scale = 1.5] at (-10,3) {$\Omega$};
		\node [scale = 1.5] at (-8,-0.3) {$\Sigma_m$};
		\draw (-8,0) to (-7.9,0.74);
		\node [scale = 1.5] at (-7,3.3) {$\Upsilon_m$};
		\draw (-7.5,2.35) to (-7.3,3);
		\node [scale = 1.5] at (1,1.5) {$\D$};
		\node [scale = 1.5] at (-1,-0.5) {$\tilde{\Upsilon}_m$};
		\draw (-1,-0.5) to (-0.3,0.2);
\end{tikzpicture}}
\caption{Isometric collar neighbourhoods of $\Sigma_m$ and $\S^1$. In the outer,
darker region $\tilde \Upsilon_m$ of $\D$, the metric is isometric to the
metric on $\Upsilon_m$. In the inner, lighter region, it is the Euclidean
metric. In the intermediate region, it is a convex combination of both the
pullback of a metric in the intermediate region on $\Omega$, and the Euclidean
metric.}
\end{figure}
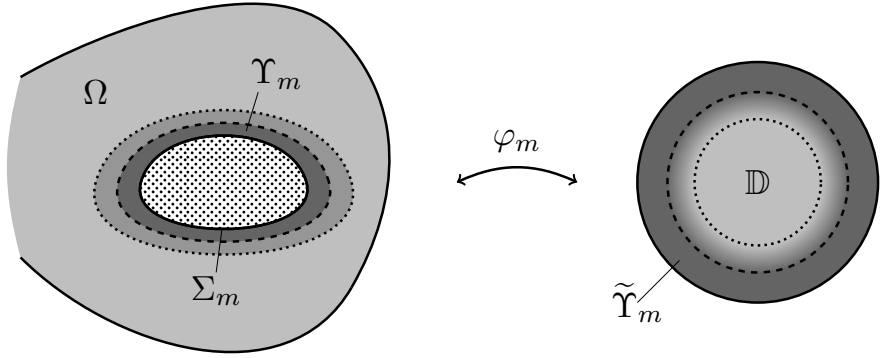
This leads us to the main theorem of this subsection, reducing the problem to
the one on the unit disk.
\begin{thm} \label{thmreduction}
Let $(\Omega,g)$ be a compact Riemannian surface whose smooth boundary $\Sigma$ has $\ell$ connected components $\Sigma_1, \dots, \Sigma_\ell$, and let $\Omega'$ be the disjoint union of $\ell$ identical unit disks
$\D_1,\dots, \D_\ell$ with boundary $\Sigma' = \S^1_1 \sqcup \dotso \sqcup \S^1_\ell$. There exist
\begin{equation}
\tau_0 : \Omega' \to \C \qquad \text{and} \qquad \rho_0 : \Sigma' \to \C.
\end{equation}
such that 
\begin{equation}
\sigma_j(\Omega;\tau;\rho;\lambda) \sim
\sigma_j(\Omega';\tau_0;\rho_0;\lambda) .
\end{equation}
\end{thm}
\begin{proof}
The proof follows that of \cite[Theorem 1.4]{GPPS}. For $1 \le m \le \ell$, let
$\Omega_m$ be a topological disk with a Riemannian metric that is isometric to a
collar neighbourhood $\Upsilon_m$ of $\Sigma_m$, and denote by $\Omega_\sharp$
the union of the disks $\Omega_m$. We abuse notation and denote also by $\tau$
any smooth function on $\Omega_\sharp$ whose value on $\Upsilon_m$ coincides
with $\tau$ on $\Omega$. This is justified since only its value in a neighbourhood of the
boundary affects eigenvalue asymptotics. It follows from Lemma \ref{lem:isombdry}
that
\begin{equation}
\abs{\sigma_j(\Omega;\tau;\rho;\lambda) -
\sigma_j(\Omega_\sharp;\tau;\rho;\lambda)} = \bigo{j^{-\infty}}.
\end{equation}
For every $m$ the Riemann mapping theorem implies the existence of a conformal
diffeomorphism $\varphi_m : (\overline{\D_m},g_0) \to (\Omega_m,g_m)$. Given that
$\varphi_m^* g_m =
e^{2f_m} g_0$, define $\tau_0$ and $\rho_0$ for $x \in \D_m$ and $\S^1_m$
respectively as
\begin{equation}
\begin{cases}
\tau_0(x) := e^{2f_m} \tau(\varphi_m(x)); \\
\rho_0(x) := e^{f_m} \rho(\varphi_m(x)).
\end{cases}
\end{equation}
It follows from Lemma \ref{lem:todisksc} that $\frac{1}{\rho}
\DN_\lambda(\Omega_m;\tau)$ is isospectral to
$\frac{1}{\rho_0}\DN_\lambda(\overline{\D_m};\tau_0)$. The conclusion then follows
from the fact that the spectrum of the Dirichlet-to-Neumann map defined on a disjoint union of domains is the
union of their respective spectra.
\end{proof}

\subsection{The symbol of the Dirichlet-to-Neumann map on the disk}
\label{sec:symboldisk}

We now compute the full symbol of $\Lambda := \frac{1}{\rho} \DN_\lambda(\D;\tau)$
on $\S^1 = \del \D$ from the factorisation
obtained in Proposition \ref{prop:facto}. Let us introduce boundary normal
coordinates $(x,t)$ for the collar neighbourhood $\S^1 \times [0,\delta)$, for
some small but fixed $\delta$. The flat metric in these coordinates reads
\begin{equation}
g(x,t) = (1-t)^2 (\mathrm dx)^2 + (\mathrm dt)^2,
\end{equation}
and the Laplacian reads as
\begin{equation}
- \Delta =  D_t^2 + \frac{i}{1 - t} D_t + \frac{1}{(1-t)^2} D_x^2.
\end{equation}
We are therefore looking for a factorisation of the form
\begin{equation}
 - \Delta_g - \lambda \tau(x) \equiv (D_t + iE(t) - i A(x,t,D_x))(D_t + i A(x,t,D_x)) \quad \mod{\Psi^{-\infty}},
\end{equation}
where $E(t) =  (1-t)^{-1}$. 

Rearranging, this implies finding $A$ such that 
\begin{equation}
A^2(x,t,D_x) - \frac{1}{(1 - t)^2}D_x^2 + i[D_t,A] - E(t) A(x,t,D_x) + \lambda \tau(x) \equiv 0 \quad \mod{\Psi^{-\infty}},
\end{equation}
which at the level of symbols is tantamount to finding $a(x,t,\xi)$ such that
\begin{equation}\label{symbolLUgen}
\sum_{K \ge 0} \frac{1}{K!}(\del_{\xi}^K a) (D_{x}^K a) - \frac{\xi^2}{(1-t)^2} + \del_{t} a - \frac{a}{1 - t} + \lambda \tau = 0,
\end{equation}
where
\begin{equation}
  a(x,t,\xi) \sim \sum_{m\leq 1} a_m(x,t,\xi)
\end{equation}
is the symbol of $A$ and the coefficients $a_m$ are positively homogeneous of degree $m$ in $\xi$.

By gathering the terms of degree two, we obtain
\begin{equation}
a_1 = -\frac{\abs{\xi}}{1 - t},
\end{equation}
while gathering the terms of degree one yields
\begin{equation}
a_0(x,t,\xi) = \frac{-1}{2a_1}\left(\del_t a_1 - \frac{a_1}{1-t}\right) = 0.
\end{equation}
One can observe that neither $a_1$ nor $a_0$ depend on $\lambda \tau$. However, by gathering the terms of order 0, we get
\begin{equation}
a_{-1}(x,t,\xi) = \frac{-\lambda \tau}{2a_1} = \frac{\lambda(1-t) \tau}{2\abs{\xi}}.
\end{equation}
For $m \leq -1$, $a_{m-1}$ is found recursively by gathering the terms of order $m$ and is given by
\begin{equation}\label{eq:am}
a_{m-1}(x,t,\xi) = -\frac{1}{2a_1}\left(\sum_{\substack{j,k \\ m \leq j, k \leq
1 \\ \gamma = j+k-m}} \frac{1}{\gamma!} D_{\xi}^\gamma (a_j) \del_{x}^\gamma(a_k) + \del_{t}a_m - \frac{a_m}{1-t} \right).
\end{equation}
Note that this is the same recurrence relation as the one appearing in
\cite{leeuhlmann} as soon as $m < -1$. For the sequel, we will require explicit knowledge of the term of order $-2$. From the previous equation we deduce that
\begin{equation}
 a_{-2}(x,t,\xi) = \frac{(1-t)\lambda}{4 \abs \xi^2} \left(i \tau_x \sgn(\xi) - 2 \tau +(1-t)\tau_t\right).
\end{equation}

As indicated by Corollary \ref{cor:bdry}, the symbol of $\Lambda$ is given by
$$r(x,\xi) = -\rho(x)^{-1}a(x,0,\xi)$$ where the sign is chosen so that $\Lambda$
is a positive operator. Note that $\del_t$ is the interior normal derivative
hence $\del_t = - \del_\nu$. Writing $f(x) := f(x,0)$ for the restriction of any
function to the boundary, the first few terms of the symbol of $\Lambda$ read
as
\begin{equation}\label{symbolLU}
r(x,\xi) = \frac{\abs \xi}{\rho(x)} - \frac{\lambda \tau(x)}{2 \rho(x) \abs \xi} + r_{-2}(x,\xi) + \bigo{\abs \xi^{-3}},
\end{equation}
with
\begin{equation}\label{symbolLU2}
r_{-2}(x,\xi) = \frac{-\lambda}{4 \rho(x) \abs \xi^2} \left(i \tau_x(x) \sgn(\xi) - 2 \tau(x) -  \del_\nu \tau(x)\right).
\end{equation}

\subsection{Symmetries of the symbol}

When $\lambda$ and $\tau$ are real, we see from these first expressions, that the real part of the symbol is an even function of $\xi$, while its imaginary part is an odd function of $\xi$. This is equivalent to the following definition.
\begin{defi}\label{def:hermitian}
A symbol $a(x,\xi)$ is hermitian if $a(x,-\xi) = \overline{a(x,\xi)}$ for all $x,\xi \in \R$.
\end{defi}

We now show recursively that the symbol of $\Lambda$ is hermitian.

\begin{prop}\label{propstructure}
  For $\lambda \in \R$, $\tau \in C^\infty(\D;\R)$, the symbol $r_m$ is
  hermitian for all $m \leq 1$.
\end{prop}

The proposition follows from \eqref{eq:am} and the following lemma whose proof is straightforward.

\begin{lem}\label{lem:herm}
Let $a$ and $b$ be two hermitian symbols corresponding to operators $A$ and $B$. Then
\begin{enumerate}
\item $\del_x a$ and $D_\xi a$ are hermitian;
\item $a + b$ and $ab$ are hermitian;
\item The symbol of $AB$ is hermitian.
\end{enumerate}

\begin{proof}
  The first two claims are a trivial computation. The third claim follows from the
  fact that the symbol of $AB$ is obtained from $a$ and $b$ using the operations described by the first two claims.
\end{proof}

\end{lem}

\section{Transformation of the symbol}\label{sec:trans}

In this section, we follow and make explicit the strategy laid out in
\cite{rozenblumalmostsim}, \cite[Section 2]{agranovichrus} and
\cite{egorovschultze} in the specific case of the parametric Dirichlet-to-Neumann map.

Specifically, we want to find a sequence $P_N \in \Psi^1$ such that
\begin{itemize}
 \item $\Lambda U_N = U_N P_N  \mod{\Psi^{1 - N}}$ for a bounded operator $U_N$;
 \item The symbol of $P_N$ depends only on the cotangent variable $\xi$ up to order $1 - N$.
\end{itemize}
Such a procedure (making the symbol dependent solely on $\xi$) will be referred
to as a \emph{diagonalisation} of the symbol. It is motivated by the following
proposition resulting from \cite[Theorem 9]{rozenblumalmostsim}.

\begin{prop}\label{propdiag}
  Let $A$ be an elliptic, self-adjoint pseudodifferential operator of order $1$
  and let $P$ be the operator with symbol
\begin{equation}
p(x,\xi) = \sum_{m=0}^N p_{1-m}(\xi)
\end{equation}
where $p_{1-m}$ depends only on $\xi$ and is positively homogeneous of order
$1-m$. Suppose that $AU - UP \in \Psi^{-N}$ for some bounded operator $U$. Then
the eigenvalues of $A$ are given by the union of two sequences, $\{\sigma_j^+\}$ and $\{\sigma_j^-\}$, that satisfy
\begin{equation}
\sigma_j^{\pm} = \sum_{m=0}^N p_{1-m}(\pm j) + \bigo{j^{-N}}.
\end{equation}
\end{prop}

\subsection{Diagonalisation of the principal symbol}
We start by diagonalising the principal symbol of $\Lambda = \frac{1}{\rho} \DN_\lambda(\D; \tau)$. Let
\begin{equation}
  \label{eq:defL}
 L = \frac{1}{2 \pi }\int_0^{2\pi} \rho(x) \de x
\end{equation} 
and
\begin{equation}
 S(x,\eta) = \frac{\eta}{L} \int_0^x \rho(t) \de t.
\end{equation}
The function $S$ is a generating function for the canonical transformation $(y, \xi) = T(x,\eta)$ given by the relations
\begin{equation}
  \xi = \frac{\del S}{\del x}, \qquad y = \frac{\del S}{\del \eta}.
\end{equation}
We define the Fourier integral operator $\Phi$ with phase function $S$ as
\begin{equation}
\Phi u (x) = \int_\R e^{i S(x,\xi)} \hat u(\xi) \de \xi,
\end{equation} 
where $\hat u$ is the Fourier transform of $u$. We use $\Phi$ to diagonalise the
principal symbol of $\Lambda $ in the following proposition.
\begin{prop} \label{propprincipalsymbol}
For any $N$, there is an operator $B_N \in \Psi^1$ such that its principal symbol depends only on $\xi$ and such that
\begin{equation}
\Lambda \Phi - \Phi B_N \in \Psi^{1 - N}.
\end{equation}
\end{prop}
\begin{proof}
We are looking for the symbol of $B$ in the form
\begin{equation}
 b(x,\xi) = b_1(\xi) + \sum_{m \le 0} b_m(x,\xi)
\end{equation}
with $b_j(x,\xi)$ positively homogeneous of order $j$ in $\xi$. Let us first study the operator $\Lambda \Phi$. It acts on smooth functions as
\begin{equation}
\begin{aligned}
  \Lambda \Phi u(x) &= \frac{1}{2\pi}\iiint r(x,\eta) e^{i(x - y) \eta} e^{i S(y,\xi)} \hat u(\xi) \de \eta \de y \de \xi \\
  &= \frac{1}{2\pi}\int k(x,\xi) \ e^{i S(x,\xi)} \hat u (\xi) \de \xi,
 \end{aligned}
\end{equation}
where
\begin{equation}
 k(x,\xi) = \iint r(x,\eta) e^{i(x - y)\eta} e^{i (S(y,\xi) - S(x,\xi))} \de y \de \eta.
\end{equation}
We now look for the asymptotic expansion of $k$ as a symbol on $\S^1$, up to
symbols of order $-\infty$. Note that the expressions here have sense in terms
of distributions, see \cite[Section 2.2.2]{egorovschultze}. By following the method of proof in
\cite[Theorem 6.5]{egorovschultze},
we can localise the symbol by finding smooth cut-off functions $h_1(x,y)$ and
$h_2(\xi, \eta)$ supported in suitable neighbourhoods of $x=y$ and $\xi = \eta$
such that if
\begin{equation}
k'(x,\xi) = \iint r(x,\eta) e^{i(x-y)\eta} e^{i(S(y,\xi) - S(x,\eta))} h_1(x, y) h_2(\xi, \eta) \de y \de \eta,
\end{equation}
then $\Op(k-k') \in \Psi^{-\infty}$. By Taylor's theorem, we can write
\begin{equation}
S(y,\xi) - S(x,\xi) = \pardiff{S(x,\xi)}{x} (y-x) + R(x,y,\xi)(y-x)^2
\end{equation}
with
\begin{equation}\label{eq:R}
R(x,y,\xi) = \int_0^1 (1-t) \pardiff{^2 S(x + t(y-x),\xi)}{x^2} \de t.
\end{equation}
We can rewrite $k'$ as
\begin{equation}
k'(x,\xi) = \iint r(x,\eta) e^{i(x - y)(\eta - R(x,y,\xi)(y-x) - \pardiff{S}{x})} h_1(x,y) h_2(\xi,\eta) \de y \de \eta.
\end{equation}
Changing variables as $\tilde{\eta} = \eta - R(x,y,\xi)(y-x)$ and $\tilde{\xi} = \pardiff{S(x,\xi)}{x} = \frac{\xi \rho}{L}$, we obtain that $k'$ is of the form
\begin{equation}
k'(x,\xi) = \iint K(x,y,\tilde{\xi},\tilde{\eta}) e^{i(x-y)(\tilde{\eta} - \tilde{\xi})} \de y \de \tilde{\eta}
\end{equation}
where
\begin{equation}
  K(x,y,\tilde{\xi},\tilde{\eta}) = r\left(x, \tilde{\eta} +
  R(x,y,\xi)(y-x)\right) h_1(x,y) h_2(\xi, \tilde{\eta} + R(x,y,\xi)(y-x)).
\end{equation}
From \cite[Lemma 2.13]{egorovschultze}, we know that $k'(x,\xi)$ is a symbol given by
\begin{equation}
k'(x,\xi) = \sum_{\alpha \geq 0} \frac{1}{\alpha !} \del_{\tilde{\eta}}^\alpha
D_y^\alpha K(x,y,\tilde{\xi},\tilde{\eta}) \bigg|_{\substack{\tilde{\eta} =
\tilde{\xi} \\ y = x}}.
\end{equation}
By the choice of cut-off functions, when $x$ is close to $y$ and $\tilde{\eta}$
is close to $\tilde{\xi}$, we have that $h_1$ and $h_2$ are constant and equal
to one. Hence, they don't intervene in the symbol's calculation and
\begin{equation}
k'(x,\xi) = \sum_{\alpha \geq 0} \frac{1}{\alpha!} \del_{\tilde{\eta}}^\alpha
D_y^\alpha r(x,\tilde{\eta} + R(x,y,\xi)(y-x))\bigg|_{\substack{\tilde{\eta} =
\tilde{\xi} \\ y = x}}.
\end{equation}

We now make the following observation : if $r(x,\tilde \eta + R(x,y,\xi)(y-x))$
is a symbol of order $m$, then applying $\del_{\tilde \eta}^\alpha D_y^\alpha$
results in a symbol of order $m - \alpha$. In fact, for $\alpha = 1$, and
denoting by $\del_2$ the derivative with respect to the second argument, we have
\begin{equation}
\begin{aligned}
\del_{\tilde \eta} D_y r(x,\tilde \eta +
R(x,y,\xi)(y-x))\bigg|_{\substack{\tilde \eta = \tilde \xi \\ y = x}} &= - i
\left[\del_2^2 r(x,\tilde \xi)\right] R(x,x,\xi) \\
&= - i \left[\del_2^2 r\left(x,\frac{\rho(x) \xi}{L}\right)\right] \frac{\xi \rho'(x)}{2L}.
\end{aligned}
 \end{equation}
It is clear from this last equation that it is a symbol of order $m - 1$.
Induction on $\alpha$ is then straightforward. This yields the asymptotic
symbolic expansion $k'(x,\xi) = \sum_{m \leq 1} \tilde{a}_m(x,\xi)$ where
\begin{equation} \label{eq:atilde}
 \tilde a_m(x,\xi) = \sum_{0 \le \alpha \le 1-m} \frac{1}{\alpha!} \del_{\tilde
 \eta}^\alpha D_y^\alpha r_{m + \alpha}(x,\tilde \eta + R(x,y,\xi)(y-x)) \bigg|_{\substack{\tilde \eta = \frac{\xi \rho(x)}{L} \\ y = x}}.
\end{equation}
We can compute the first few terms of the symbolic expansion, using the fact
that in $\R \setminus\set 0$ the second derivative of $a_1$ in the second
variable vanishes identically. This gives
\begin{align*}
 \tilde a_1(x,\xi) &= \frac{\abs\xi}{L};\\
 \tilde a_0(x,\xi) &= 0 ;\\
 \tilde a_{-1}(x,\xi) &= -\frac{\lambda L \tau}{2 \rho^2 \abs \xi};\\
 \tilde a_{-2}(x,\xi) &= \frac{\lambda L^2}{4 \xi^2 \rho^3}\left(\tau_r - i\sgn(\xi) \tau_x + 2\tau\right) + \frac{i \lambda L^2 \tau \sgn(\xi) \rho'}{2\xi^2 \rho^4}.
\end{align*}

Let us now compute the symbol of $\Phi B$. We have
\begin{equation}
\begin{aligned}
  \Phi B u(x) &= \frac{1}{2\pi}\iiint e^{i y(\xi - \eta)} e^{i S(x,\eta)} b(y,\xi) \hat u(\xi) \de y \de \eta \de \xi \\
  &= \frac{1}{2\pi}\int f(x,\xi) e^{i S(x,\xi)} \hat u (\xi) \de \xi,
 \end{aligned}
\end{equation}
where
\begin{equation}
 f(x,\xi) = \iint b(y,\xi) e^{i (S(x,\eta) - S(x,\xi))} e^{i y(\xi - \eta)} \de y \de \eta.
\end{equation}
As above, this integral only converges in the sense of distributions. As in
\cite{egorovschultze}, we can find a smooth cut-off function $h(\xi,\eta)$
supported in a neighbourhood of $\xi = \eta$ such that the symbol
\begin{equation}
f'(x,\xi) = \iint b(y,\xi) e^{i (S(x,\eta) - S(x,\xi))} e^{i y(\xi - \eta)} h(\xi,\eta) \de y \de \eta
\end{equation}
satisfies $\Op(f - f') \in \Psi^{-\infty}$.

Let us observe that
\begin{equation}
 S(x,\eta) - S(x,\xi) = \frac{(\eta - \xi)}{L} \int_0^x \rho(x) \de x  = (\eta -
 \xi) F(x)
\end{equation}
and that $F(x) = \frac{\del S}{\del \xi}(x,\xi)$. After the change of variable
$\tilde y = y + x - F(x)$, the equation for $f'$ becomes
\begin{equation}
\begin{aligned}
 f'(x,\xi) &= \iint b(\tilde y - x + F(x), \xi) h(\xi,\eta) e^{i (x - \tilde
 y)(\eta - \xi)} \de \tilde y \de \eta \\
 &= \iint Q(x,\tilde y, \eta, \xi) e^{i (x - \tilde y)(\eta - \xi)} \de \tilde y
 \de \eta.
 \end{aligned}
\end{equation}
Once again from \cite[Lemma 2.13]{egorovschultze}, we have that $f'$ is a symbol in $S^1$ and
\begin{equation}
f'(x,\xi) = \sum_{\alpha \ge 0} \frac{1}{\alpha!} \del_{\eta}^\alpha D_{\tilde
y}^\alpha Q(x,\tilde y, \xi,\eta) \bigg|_{\substack{\tilde y = x \\ \eta =
\xi}}.
\end{equation}
Since $Q$ is constant in $\eta$ close to $\xi$, the derivatives in $\eta$ always vanish.
Hence, the symbol of $B_N$ is given by
\begin{equation}
f'_N(x,\xi) = \sum_{-N \le m \le 1} b_m\left(\frac{1}{L}\int_0^x \rho(t) \de t,\xi \right).
\end{equation}
To have the terms of the same order of homogeneity cancel out, we need to choose
\begin{equation}\label{eq:bm}
 b_m(x,\xi) = \tilde a_m(s(x),\xi),
\end{equation}
where $s(x)$ is the number $s$ such that
\begin{equation}
x = \frac{1}{L} \int_0^s \rho(t) \de t.
\end{equation}
This concludes the proof.
\end{proof}

\subsection{Diagonalisation of the full symbol} \label{subsecdiagsub}
Let us denote by $P_1$ the operator with symbol
\begin{equation}
  \label{eq:p(1)}
p^{(1)}(x,\xi) = b_1(\xi) + \sum_{m \leq -1} b_m(x,\xi).
\end{equation}
The diagonalisation of the full symbol is based on the following lemma inspired
by the methods laid out by Rozenblum \cite{rozenblumalmostsim} and Agranovich
\cite{agranovichrus}. We include it for completeness.
\begin{lem}\label{diagsub}
Let $N \geq 0$ and suppose that there exists a bounded operator $U_N$ such that
$\Lambda U_N - U_N P_N \in \Psi^{-\infty}$ where $P_N$ is a pseudodifferential
operator whose symbol is given by
\begin{equation}
p^{(N)}(x,\xi) = \sum_{m=0}^N p^{(N)}_{1-m}(\xi) + p^{(N)}_{-N}(x,\xi) + \bigo{\abs{\xi}^{-(N+1)}}.
\end{equation}
Then if
\begin{equation}\label{eq:pNtilde}
p^{(N+1)}_{-N}(\xi) = \frac{1}{2\pi} \int_0^{2\pi} p^{(N)}_{-N}(x,\xi) \de x
\end{equation}
and $K$ is the pseudodifferential operator with symbol
\begin{equation}\label{eq:k}
k(x,\xi) = 1 - iL\sgn \xi \int_0^x p^{(N)}_{-N}(t,\xi) - p^{(N+1)}_{-N}(\xi) \de t,
\end{equation}
there exists an operator $P_{N+1}$ with symbol
\begin{equation}
p^{(N+1)}(x,\xi) = \sum_{m=0}^N p^{(N)}_{1-m}(\xi) + p^{(N+1)}_{-N}(\xi) + \bigo{\abs{\xi}^{-(N+1)}}
\end{equation}
satisfying $\Lambda (U_N K) - (U_N K) P_{N+1} \in \Psi^{-\infty}$.
\end{lem}

\begin{proof}
Starting off with the pseudodifferential operator $P_N$, we would like to find a
bounded operator $K$ and a pseudodifferential operator $P_{N+1}$ whose symbol
$p^{(N+1)}$ satisfies
\begin{equation}
p^{(N+1)}(x,\xi) = \sum_{m=0}^N p_{1-m}^{(N)}(\xi) + p^{(N+1)}_{-N}(\xi) +
\bigo{\abs{\xi}^{-(N+1)}}
\end{equation}
such that $P_N K - K P_{N+1} \in \Psi^{-\infty}$. We choose $K$ to have symbol
$1 + k_{-N}(x,\xi)$ with $k_{-N}$ positively homogeneous of order $-N$ in $\xi$.
The symbol of $P_N K - K P_{N+1}$ is then given by
\begin{equation}
p_{-N}^{(N)}(x,\xi) - p_{-N}^{(N+1)}(\xi) - i(\del_\xi p_1^{(N)})(\del_x k_{-N}) + O(\abs{\xi}^{-N-1}).
\end{equation}
The symbol $p_1^{(N)}$ comes from the diagonalisation of the principal symbol
and is given by $p_1^{(N)}(\xi) = p_1^{(1)}(\xi) = b_1(\xi) =
\frac{\abs{\xi}}{L}$. Hence, we see that the terms of order $-N$ cancel if the
symbol of $K$ is given by \eqref{eq:k} and since $0 = k_{-N}(0,\xi) =
k_{-N}(2\pi,\xi)$, we must take $p_{-N}^{(N+1)}$ as in \eqref{eq:pNtilde}. In
order to get that $P_N K - K P_{N+1} \in \Psi^{-\infty}$ knowing that the symbol
of $P$ is given by
\begin{equation}
p^{(N)}(x,\xi) = \sum_{m=0}^N p^{(N)}_{1-m}(\xi) + \sum_{m \geq N+1}
p^{(N)}_{1-m}(x,\xi),
\end{equation}
we need to take $P_{N+1}$ with symbol
\begin{equation}
p^{(N+1)}(x,\xi) = \sum_{m=0}^N p^{(N+1)}_{1-m}(\xi) + p^{(N+1)}_{-N}(\xi) +
\sum_{m \geq N+2} p^{(N+1)}_{1-m}(x,\xi),
\end{equation}
which is calculated inductively as
\begin{equation}\label{eq:pmtilde}
p^{(N+1)}_m = p^{(N)}_m + \sum_{\alpha = 0}^{1-m-N} \frac{1}{\alpha !} \left[(\del_x^\alpha k_{-N})(D_{\xi}^{\alpha} p^{(N)}_{m + \alpha + N}) - (\del_x^\alpha p^{(N+1)}_{m + \alpha + N})(D_\xi^\alpha k_{-N})\right]
\end{equation}
for $m \leq -N-1$. It follows that $\Lambda(U_N K) - (U_N K) P_{N+1}$ is smoothing.
\end{proof}

The previous lemma gives us a family of operators $P_N$ that diagonalise
$\Lambda$ down to any desired order. By applying it $N-1$ times starting from
$P_1$, we get that there exists $P_N$ with symbol
\begin{equation}
p^{(N)}(x,\xi) = \frac{\abs{\xi}}{L} + \sum_{m=1}^{N-1} \frac{1}{2\pi}
\int_0^{2\pi} p_{-m}^{(m)}(x,\xi) \de x + \bigo{\abs{\xi}^{-N}}
\end{equation}
such that $\Lambda U_N - U_N P_N$ is smoothing for some bounded operator $U_N$.
We summarise the properties of the operators $P_N$ that were proved along the
discussion above in the following
proposition.

\begin{prop}\label{prop:propertiespN}
The symbols $p^{(N)}$ of $P_N$ possess the following properties.
\begin{enumerate}
  \item  The first symbol $p^{(1)} = b_1(\xi) + \sum_{m \leq -1} b_m(x,\xi)$,
    see \eqref{eq:p(1)}.
\item  For $m \geq 1-N$, $p_m^{(N+1)} = p_m^{(N)}$ and $\del_x p_m^{(N)} = 0$.
  In other words, for every $m \le -1$ the sequence stabilises and eventually becomes diagonal with
  respect to $\xi$.
\item  For $m \leq -N-1$, $p_m^{(N+1)}$ is given recursively by equation \eqref{eq:pmtilde}.
\item  When the sequence stabilises, the diagonalised symbol can be explicitly
  computed as $p_{-N}^{(N+1)}(\xi) = \frac{1}{2\pi} \int_0^{2\pi} p_{-N}^{(N)} (x,\xi) \de x$.
\end{enumerate}
\end{prop}

One can see that $p_m^{(N)}$ is a polynomial in $\lambda$ with coefficients that
are functions of $x$ and $\xi$. From this point of view, we observe the following.

\begin{lem}\label{lem:deglambda}
For each $m \leq -1$ and for each $N \geq 1$, the function $p_m^{(N)}$ is a polynomial in $\lambda$ of degree at most $-m$ whose constant coefficient vanishes.
\end{lem}

\begin{proof}
We denote by $\deg(p)$ the degree of a function $p(x,\xi)$ as a polynomial in $\lambda$. We proceed by induction on both $N$ and $m$.

It is easily seen from \eqref{eq:am} and the expressions for $a_{1}$ and $a_{-1}$ that the functions $a_m$ (and hence $r_m$) are polynomials of order $\ceil{\frac{-m}{2}} \leq -m$ whenever $m \leq -1$. It then follows from equations \eqref{eq:atilde} and \eqref{eq:bm} that $\deg(p_m^{(1)}) = \deg(b_m) \leq -m$ for all $m \leq -1$.

Let $N \geq 1$ be arbitrary and suppose that $\deg(p_m^{(N')}) \leq -m$ for all $1 \leq N' \leq N$ and $m \leq -1$. From Proposition \ref{prop:propertiespN}, we know that
\begin{equation}
\deg(p_{-1}^{(N+1)}) = \deg(p_{-1}^{(1)}) = 1.
\end{equation}
Let $m_0 \leq -1$ and suppose that $\deg(p_{m}^{(N+1)}) \leq -m$ for all $-1 \geq m \geq m_0$. We want to estimate the degree of $p_{m_0-1}^{(N+1)}$. Its expression is given by \eqref{eq:pmtilde} and we can see that the term of highest degree in $\lambda$ in the sum is obtained when $\alpha = 0$. Hence,
\begin{equation}\label{eq:deg1}
\deg(p_{m_0-1}^{(N+1)}) \leq \deg(k_{-N}) + \deg(p_{m_0-1+N}^{(N+1)})
\end{equation}
From the definition of $k_{-N}$, we have 
\begin{equation}\label{eq:deg2}
\deg(k_{-N}) = \deg(p_{-N}^{(N)}) \leq N
\end{equation}
by the induction hypothesis. Since $m_0 - 1 + N \geq m_0$, the induction hypothesis yields 
\begin{equation}\label{eq:deg3}
\deg(p_{m_0-1+N}^{(N+1)}) \leq -m_0 + 1 -N.
\end{equation}
Therefore, by combining \eqref{eq:deg1}, \eqref{eq:deg2} and \eqref{eq:deg3}, $\deg(p_{m_0-1}^{(N+1)}) \leq -m_0 + 1$ and the claim follows by induction.

Finally, to show that the constant coefficient of $p_m^{(N)}$ vanishes, it
suffices to show that it is the case for $a_m$. Proceeding inductively, since
$a_0 = 0$, notice from \eqref{eq:am} that the only term in $a_{m-1}$ that could
be constant in $\lambda$ is $\frac{1}{\gamma!} D_\xi^{\gamma}(a_1) \del_x^\gamma(a_1)$ with $\gamma = 2-m$. However, $D_\xi^{\gamma}(a_1) = 0$ for $\gamma \geq 2$.
\end{proof}

\begin{rem}
  That $p_m^{(N)} = 0$ whenever $\lambda = 0$ is not surprising. Indeed, this
corresponds to the classic Dirichlet-to-Neumann operator whose symbol is
precisely $\abs{\xi}$.
\end{rem}

If one is interested in computing the symbols explicitly in a given example the
calculations quickly become very involved. The following lemma allows us to
reduce the number of computations to obtain the $k$-th term in the diagonalised
symbol.

\begin{lem}\label{lem:simple}
For all $N \geq \ceil{\frac{-m}{2}}$,
\begin{equation}
\int_0^{2\pi} p^{(-m)}_{m} \de x = \int_0^{2\pi} p^{(N)}_{m} \de x.
\end{equation}
\end{lem}

\begin{proof}
If $m \geq 1-2N$, then $\del_x^{\alpha} p^{(N)}_{m+\alpha+N} = \del_x^{\alpha} p^{(N+1)}_{m+\alpha+N} = 0$ for all $\alpha > 0$. We also have $p^{(N)}_{m+N} = p^{(N+1)}_{m+N}$ and hence
\begin{equation}
p_{m}^{(N+1)} = p_m^{(N)} + \sum_{\alpha = 1}^{1-m-N} \frac{1}{\alpha!}
(\del_x^\alpha k_{-N})(D_\xi^{\alpha} p_{m+\alpha+N}^{(N)}).
\end{equation}
Therefore, since $p_{m+\alpha+N}^{(N)}$ doesn't depend on $x$, integrating both sides yields
\begin{equation}
\int_{0}^{2\pi} p_{m}^{(N+1)} \de x = \int_{0}^{2\pi} p_{m}^{(N)} + \sum_{\alpha = 1}^{1-m-N} \frac{1}{\alpha!} (D_{\xi}^{\alpha} p_{m+\alpha+N}^{(N)}) \int_{0}^{2\pi} (\del_x^\alpha k_{-N}) \de x.
\end{equation}
The rightmost integral vanishes for all $\alpha$ since $k_{-N}$ is periodic and thus
\begin{equation}
\int_{0}^{2\pi} p_m^{(N+1)} \de x = \int_{0}^{2\pi} p_m^{(N)}\de x.
\end{equation}
Finally, if $m = -2N$, we have
\begin{equation}
\int_0^{2\pi} p_{-2N}^{(N+1)} \de x = \int_0^{2\pi} p_{-2N}^{(N)} \de x + \int_0^{2\pi} k_{-N}(p_{-N}^{(N)} - p_{-N}^{(N+1)}) \de x
\end{equation}
and since $\del_x k_{-N} = -iL\sgn\xi(p_{-N}^{(N)} - p_{-N}^{(N+1)})$ the rightmost integral vanishes. The result then follows since $m \geq -2N$ is equivalent to $N \geq \ceil{\frac{-m}{2}}$.
\end{proof}

The previous lemma simplifies calculations. Indeed, in order to get the
diagonalised term of order $-m$, it suffices to apply the diagonalisation lemma
$\ceil{\frac{m}{2}}$ rather than $m$ times. In particular, we get
\begin{equation}
  \begin{aligned}
\int_0^{2\pi} p_{-2}^{(2)} \de x &= \int_0^{2\pi} p_{-2}^{(1)} \de x \\
&= \int_0^{2\pi} b_{-2}(x,\xi) \de x\\
&= \int_0^{2\pi} \tilde{a}_{-2}(s(x),\xi) \de x.
\end{aligned}
\end{equation}
Using that $s'(x) = \frac{L}{\rho(s(x))}$, we get
\begin{equation}
  \begin{aligned}
\int_0^{2\pi} p_{-2}^{(2)} \de x &= \frac{1}{L} \int_0^{2\pi} \rho(x)
\tilde{a}_{-2}(x,\xi) \de x \\
&= \frac{\lambda L}{4\abs{\xi}^2} \int_0^{2\pi} \frac{\tau_r + 2\tau}{\rho^2} \de x
\end{aligned}
\end{equation}
where the terms containing $i \sgn \xi$ vanish from the fact that
\begin{equation}
\int_0^{2\pi} \frac{\tau_x}{\rho^2} \de x = 2 \int_0^{2\pi} \frac{\tau
\rho'}{\rho^3} \de x,
\end{equation}
this equality being obtained by integrating by parts.
Therefore, by doing a similar calculation for $\int_{0}^{2\pi} b_{-1}(x,\xi) \de
x$, we see that the symbol of $P_2$ is given by
\begin{equation}\label{eq:p3}
p^{(2)}(x,\xi) = \frac{\xi}{L} - \frac{\lambda}{4\pi\abs{\xi}} \int_0^{2\pi} \frac{\tau}{\rho} \de x + \frac{\lambda L}{8\pi \abs{\xi}^2} \int_0^{2\pi} \frac{\tau_r + 2 \tau}{\rho^2} \de x + \bigo{\abs{\xi}^{-3}}.
\end{equation}

\section{General eigenvalue asymptotics from the symbol}\label{sec:asymp}

\subsection{Self-adjointness}

For $\lambda \in \R\cap \CV$ and $\tau$ real-valued, the operator $\Lambda :=
\frac{1}{\rho}\DN_\lambda(\D; \tau)$ is self-adjoint and therefore has real
spectrum. This follows from the fact that $\DN_\lambda(\D; \tau)$ is
self-adjoint and the following lemma applied to $P = \DN_\lambda(\D;\tau)$.

\begin{lem}
 Let $P$ be a self-adjoint pseudodifferential operator on $L^2(\S^1;\de x)$ and
 $\rho > 0$ be a positive function on $\S^1$ and denote $M_{1/\rho}$ the
 operator of multiplication by $\rho^{-1}$. For $f \in
 \operatorname{Diff}(\S^1)$, define by $K_f$ the composition operator $K_f u = u
 \circ f$. Defining 
\begin{equation}
  g(x) = \frac{1}{L} \int_0^x \rho(t) \de t \in \operatorname{Diff}(\S^1),
\end{equation}
 the operator
 \[
  Q = K_g^{-1} M_{1/\rho} P K_g
\]
is self-adjoint on $L^2(\S^1; \de x)$. 
\end{lem}

\begin{proof}
 The operator $K_g$ is an invertible isometry from $L^2(\S^1;\de x)$ to $L^2(\S^1;\rho(x)/ L \de x)$. Indeed, for $u,v \in L^2(\S^1; \de x)$, we have
\begin{align*}
 (K_g u,K_g v)_{L^2(\rho(x)/L \de x)} &= \int_0^{2\pi} u(g(x)) v(g(x)) g'(x) \de x \\ 
 &= \int_0^{2\pi} u(x) v(x) \de x \\
 &= (u,v)_{L^2(\de x)}.
\end{align*}
 The operator $M_{1/\rho} P$ is self adjoint on $L^2(\S^1; \rho(x)/L \de x)$, hence we have
 \begin{align*}
  (u, Qv )_{L^2(\de x)} &= (u,K_g^{-1} M_{1/\rho} P K_g v)_{L^2(\de x)} \\&= (K_g u, M_{1/\rho} P K_g v)_{L^2(\rho(x)/L \de x)} \\
  &= (M_{1/\rho} P K_g u, K_g v)_{L^2(\rho(x)/L \de x)} \\
  &= (K_{g}^{-1} M_{1/\rho} P K_g u,v)_{L^2(\de x)} \\
  &= (Qu,v)_{L^2(\de x)},
 \end{align*}
proving that $Q$ is self adjoint.
\end{proof}

\subsection{General eigenvalue asymptotics}
We have shown how to diagonalise the symbol down to any order. We can now deduce
the spectral asymptotics of $\Lambda$ from Proposition \ref{propdiag}.
Eigenvalue asymptotics for an elliptic pseudodifferential operator on a circle
are discussed also in \cite[Theorem 3.1]{agranovichrus}.
\begin{thm}\label{thm:rhotau}
The eigenvalues of $\Lambda$ are asymptotically double and admit a full
asymptotic expansion given by
\begin{equation}
\sigma_{2j} \sim \sigma_{2j-1} \sim \frac{j}{L} +
\sum_{n=1}^{\infty} \frac{1}{2\pi j^n} \int_0^{2\pi} p_{-n}^{(n)}(x,1) \de x
\end{equation}
for all $N \geq 0$. Truncating the series to its first two terms, this yields
\begin{equation}
  \label{eq:shortasymp}
  \sigma_{2j} = \frac{j}{L} - \frac{\lambda}{4\pi j} \int_{\S^1} \frac{\tau}{\rho} \de x +
\frac{\lambda L}{8\pi j^2} \int_{\S^1} \frac{\tau_r + 2 \tau}{\rho^2} \de x +
\bigo{j^{-3}}.
\end{equation}

\end{thm}

\begin{proof}
The fact that the eigenvalues admit a complete asymptotic expansion follows from
Proposition \ref{propdiag} and Lemma \ref{diagsub}. Moreover, \eqref{eq:shortasymp}
follows from equation \eqref{eq:p3} and Proposition \ref{propdiag} . It
remains to show that the eigenvalues are asymptotically double. This will follow
from Proposition \ref{propdiag} if we can show that, for all $N \in \N$, there
exist a bounded operator $U_N$ and a pseudodifferential operator $P_N$ with
symbol
\begin{equation}
p(x,\xi) = \sum_{m=0}^N p_{1-m}(\xi) + \bigo{\abs{\xi}^{-N}}
\end{equation}
such that $p_{1-m}$ is an even function of $\xi$ (since then $p_{1-m}(j) =
p_{1-m}(-j)$) and such that $\Lambda U_N - U_N P_N$ is smoothing. To do so, it is
sufficient to show that a symbol being hermitian is an invariant
property of the diagonalisation procedure, see Definition \eqref{def:hermitian}.
The claim will then follow since $\Lambda$ is self-adjoint and hence all its
eigenvalues must be real.

We know from Proposition \ref{propstructure} that the symbol of $\Lambda$ is
hermitian. In order to diagonalise the principal symbol, we conjugated by the
Fourier integral operator $\Phi$. The resulting symbol is given by
\begin{equation}
b(x,\xi) \sim \sum_{m \leq 1} \tilde{a}_m(s(x),\xi)
\end{equation}
where $\tilde{a}_m$ is given by \eqref{eq:atilde}. It suffices to show that
$\tilde{a}_m$ is hermitian for all $m$. This is a consequence of the fact that
\begin{equation}\label{eq:eventildea}
  D_{\tilde \eta}^\alpha \del_y^\alpha r_{m + \alpha}(x,\tilde \eta + R(x,y,\xi)(y-x)) \bigg|_{\substack{\tilde \eta = \frac{\xi \rho(x)}{L} \\ y = x}}
\end{equation}
is hermitian for all $\alpha \geq 0$. Indeed, by Leibniz's formula and \eqref{eq:R} we have
\begin{equation}
\del_{y}^\beta [R(x,y,\xi)(y-x)] \big|_{y=x} = \frac{\xi}{(\beta+1) L} \rho^{(\beta)}(x)
\end{equation}
for all $\beta \geq 0$. Hermiticity of
\eqref{eq:eventildea} then
follows from Faà di Bruno's formula since each derivative in the second argument
will come with a power of $\xi$, thus preserving the parity in the real and
imaginary parts.

Let $N \geq 0$ and suppose that $\Lambda U_N - U_N P_N \in \Psi^{-\infty}$ as in
the notation of Lemma \ref{diagsub} is such that the symbol $p^{(N)}$ of
$P_N$ is hermitian. From \eqref{eq:k}, \eqref{eq:pmtilde} and Lemma
\ref{lem:herm}, we see that the symbol $p^{(N+1)}$ of $P_{N+1}$ is also
hermitian. The fact that the spectrum is asymptotically double then follows from
the previous discussion.

\end{proof}

\section{Eigenvalue asymptotics}\label{sec:Steklovasymp}

Let $(\Omega,g)$ be a simply connected compact Riemannian surface with smooth boundary
$\Sigma$. We are now interested in finding the spectral asymptotic for the
operator
$\DN_\lambda(\Omega;\tau)$ corresponding to the problem
\begin{equation}
\begin{cases}
  - \Delta_g u = \lambda \tau u & \text{in } \Omega; \\
  \del_\nu u = \sigma u & \text{on } \Sigma;
\end{cases}
\end{equation}
the parametric Steklov problem on $\Omega$. By the Riemann
mapping theorem, there exists a conformal diffeomorphism $\varphi$ which maps $(\D,g_0)$
onto $\Omega$ such that $\varphi^*g = e^{2f} g_0$ for some smooth function $f :
\D \rightarrow \R$. Therefore, the parametric Steklov problem on
$(\Omega,g)$ is isospectral to the problem
\begin{equation}
 \begin{cases}
  - \Delta u = \lambda e^{2f}\phi^*\tau u & \text{in } \D; \\
  \del_\nu u = \sigma e^{f} u & \text{on } \S^1.
 \end{cases}
\end{equation}
In the notation of \eqref{eq:defL}, we have
\begin{equation}
L = \frac{1}{2\pi} \int_{0}^{2\pi} e^f \de x = \frac{\mathrm{\per}_g(\Sigma)}{2\pi}.
\end{equation}

We are now in a position to prove our main results about eigenvalue asymptotics.

\begin{proof}[Proof of Theorem \ref{thm:sc}]
  The theorem follows directly from Theorem \ref{thm:rhotau} for the existence
  of the complete asymptotic expansion. The fact that $s_n$ is a polynomial in $\lambda$ of degree at most $n$ follows directly from Lemma $\ref{lem:deglambda}$. For the explicit values of $s_{-1}$ and
  $s_{-2}$ when $\tau \equiv 1$, we replace in \eqref{eq:shortasymp} the values of $\tau$ and $\rho$
  by the conformal factor. The second term in \eqref{eq:shortasymp} is given by
  \begin{equation}
    \frac{\lambda}{4\pi j} \int_{\S^1} e^{f} \de x = \frac{\lambda L}{2j}.
  \end{equation}
  Finally, the third term is given by $\frac{\lambda L}{8\pi j^2} (G + 4\pi)$
  where
\begin{equation}
G := \int_{\S^1} \frac{(e^{2f})_r}{e^{2f}} \de x = \int_{\S^1} \del_\nu \log
e^{2f} \de x = 2 \int_{\S^1} \del_\nu f \de x.
\end{equation}
By Green's theorem, we have
\begin{equation}
G = 2 \int_{\D} \Delta f \de A.
\end{equation}
Recall that the Gaussian curvature of $(\D, \varphi^* g)$ is given by
\begin{equation}
K_{\varphi^* g} = -e^{-2f} \Delta f.
\end{equation}
Hence, since $\varphi^* K_g = K_{\varphi^* g}$ and $\varphi^* \de A_g = e^{2f} \de A$,
\begin{equation}
  \begin{aligned}
G &= -2 \int_{\D} K_{\varphi^* g} e^{2f} \de A \\&= -2 \int_{\D} \varphi^*(K_{g} \de
A_g) \\&= -2 \int_{\Omega} K_g \de A_g.
\end{aligned}
\end{equation}
Combining everything and using the Gauss-Bonnet theorem yields
\begin{equation}
  \begin{aligned}
\frac{\lambda L}{8\pi j^2} (G + 4 \pi) &= \frac{\lambda L}{4\pi j^2}\left(2\pi -
\int_{\Omega} K_g \de A_g\right)\\ &= \frac{\lambda L}{4\pi j^2} \int_{\Sigma} k_g
\de s
\end{aligned}
\end{equation}
since $\Omega$ is simply connected, and hence its Euler characteristic is $1$.
\end{proof}

\begin{proof}[Proof of Theorem \ref{thm:mc}]

Let $(\Omega, g)$ now be any Riemannian surface
whose smooth boundary $\Sigma$ has $\ell$ connected
components $\Sigma_1, \dots, \Sigma_\ell$. For $1 \le m \le \ell$, let $\Omega_m$ be a smooth
topological disk with a Riemannian metric $g_m$ such that there is an isometry
$\phi_m : \tilde \Upsilon_m \to \Upsilon_m$, for collar
neighbourhoods $\tilde \Upsilon_m$ of $\del \Omega_m$ and
$\Upsilon_m$ of $\Sigma_m$. The existence of $(\Omega_m,g_m)$ is guaranteed by Lemma
\ref{lem:cutting}. Define $\tau_m : \tilde \Upsilon_m \to \R$ by $\tau_m =
\phi_m^* \tau \big|_{\Upsilon_m}$. Extend $\tau_m$ to a smooth function on
$\Omega_m$, which we still denote $\Omega_m$. This can be done, say, with an
harmonic extension and then smoothing it with a mollifier, maybe reducing a
little bit the size of the collar neighbourhoods. Denote by $\Omega_\sharp$ the
disjoint union of the disks
$\Omega_m$, $g_\sharp$ the metric which restricts to $g_m$ on every $\Omega_m$ and $\tau_\sharp \in C^\infty(\Omega)$ to be the function which
restricts to $\tau_m$ on every $\Omega_m$.

From Lemma \eqref{lem:isombdry}, we know that
\begin{equation}
  \sigma_j(\lambda, \tau, \Omega) \sim \sigma_j(\lambda, \tau_\sharp,  \Omega_\sharp).
\end{equation}
This concludes the proof of statement (A), and also implies statement (B) since
every other metric and function $\tau'$ satisfying the conclusion of statement (A) is
isometric to $\tau_\sharp$ and $g_\sharp$ in a collar neighbourhood of
$\Omega_\sharp$.

Since $\Omega_\sharp$ is a union of disks, its spectrum is given by the union of
each disk's spectrum. Applying Theorem \eqref{thm:sc} to each $\Omega_m$, and
using that the parametric Steklov spectrum of a disjoint union of surfaces is
the union of their spectra we
see that the spectrum of $\Omega$ is the union of $\ell$ different sequences
taking the form of equation \eqref{eq:aexpsc}. This is the statement (C) of Theorem
\ref{thm:mc}, which concludes its proof.

\end{proof}
\section{Geometric spectral invariants}
\label{sec:invariants}

In this section we obtain spectral invariants that have a geometric
interpretation in the case where the potential is a constant, $\tau \equiv 1$.
When the surface $\Omega$ is simply connected, the search for spectral
invariants is easier.
From the first two terms of the eigenvalue asymptotic expansion, we can deduce
uniquely the values of both $L$ and $\lambda$. Hence, from the third term, we
can deduce uniquely the value of $\int_{\Sigma} k_g \de s$ and it is a
spectral invariant. From the Gauss--Bonnet theorem, we get the following result.

\begin{cor}\label{cor:simply}
Let $(\Omega,g)$ be a simply connected compact Riemannian surface with smooth boundary
$\Sigma$. Then, the total
curvature
\begin{equation}
  \int_\Omega K_g \de A_g
\end{equation}
is a spectral invariant of the constant potential parametric Steklov problem on $\Omega$. In
particular, if the Gaussian curvature is assumed to be a constant $K(\Omega)$,
the quantity
\begin{equation}
K(\Omega) \area(\Omega)
\end{equation}
is a spectral invariant of the constant potential parametric Steklov problem on $\Omega$.
\end{cor}

In the multiply connected case, we need to introduce some definitions to talk
about functions between two multisets. To determine the number of boundary
components and the lengths of them, we will use methods from Diophantine
approximation.
This is in the spirit of \cite{GPPS}, where they obtained those quantities as invariants of the Steklov problem
with $\lambda = 0$. There, they had an asymptotic expansion of the form
\eqref{eq:sequence}--\eqref{eq:mc}, where all the coefficients $s_n$ were $0$.
However in order to obtain the number of boundary components and their lengths
as spectral invariants,
they need only that the second term is $\smallo 1$, which we do have. 

Recovering $\lambda$ as well as the total geodesic curvature of the boundary is
more complicated and requires an algorithmic procedure to recover subsequences (which
can be explicitly constructed) once we know the number of
boundary components and the length of the largest one. We start by introducing
terminology found in \cite[Section 2.3]{GPPS}

\begin{defi}
  Let $A$, $B$ be two multiset of positive real numbers. We say that $F : A \to
  B$ is \emph{close} if it has the property that for every $\eps >0$, there are
  only finitely many $x \in A$ with $\abs{F(x) - x} \ge \eps$. We say that $F$
  is an \emph{almost-bijection} if for all but finitely many $y \in B$, the
  pre-image $F^{-1}(y)$ consists in a single point.
\end{defi}
For a finite set of positive real numbers $M = \set{\alpha_1,\dotsc,\alpha_\ell}$,
we denote by $R(M)$ the multiset 
$$R(M) := \set{0,\dotsc,0} \cup \alpha_1 \N \cup \alpha_1 \N \cup
\dotso \cup \alpha_\ell \N \cup \alpha_\ell \N,
$$
where $0$ is repeated $\ell$ times and the union is understood in the sense of
multisets, i.e. multiplicity is conserved.

\begin{prop} \label{prop:decoupling}
  Let $M = \set{\alpha_1,\dotsc,\alpha_\ell}$ be a finite multi-set of positive numbers. For $N \in \N$, let
  $$\Xi^{(N)} = \set{\set{\xi^{(m,N)}_j : j \in \N} : 1 \le m \le \ell}$$ be a set of sequences given by $\xi^{(m,N)}_0 = 0$ and such that for $j \geq 1$,
  \begin{equation}
    \label{eq:aexpabstract}
    \xi^{(m,N)}_{2j} = \xi^{(m,N)}_{2j-1} + \bigo{j^{-N-1}} = j\alpha_m + \sum_{n = 1}^N s_n^{(m)} j^{-n} + \bigo{j^{-N-1}}.
  \end{equation}
  Then, $M$ and the quantities $s_n^{(m)}$ for $1 \leq n \leq N$ are uniquely determined by the sequence $S(\Xi^{(N)})$ defined as the reordering of the union of the sequences $\xi^{(m,N)}$ in increasing order.
\end{prop}

Let us first describe heuristically how the proof goes. In the first step, we
simply show that \cite[Lemmas 2.6 and 2.8]{GPPS} apply to this situation. This
will allow us to recover $M$ from $S(\Xi^{(N)})$, and we assume from then on that $M$,
and therefore $R(M)$, are already known to be spectral invariants.

In the second step, we show that for any $\alpha_m \in M$ which is not an integer
multiple of another strictly smaller element of $M$, we can identify a subsequence along which
$S(\Xi^{(N)})_j = \xi^{(m,N)}_{k(j)}$ where $k : \N \to \N$ is a function that can be
computed explicitly. For this, we use Dirichlet's simultaneous approximation
theorem.

In the third step, we obtain the coefficients of those sequences $\alpha_m$ that
we decoupled in the previous step. Obviously, if $\alpha_m$ appears only once in
$M$ this is trivial, the difficulty comes when $\alpha_m$ has
multiplicity.

In the fourth step, we proceed inductively and show that if $\alpha_m$ is an
integer multiple of some other $\alpha_n \in M$, but we already know the
coefficients of the relevant sequences for $\alpha_n$, then we can apply the
same procedures as in steps 2  and 3 to recover the coefficients of $\xi^{(m,N)}$ for any $N$.

\begin{proof}
  \textbf{Step 1: } We obtain $M$ from $S(\Xi^{(N)})$. Combining \cite{GPPS}[Lemmas
  2.6 and 2.8], as soon as $A$ is a multiset such that there exists a
  close almost-bijection $F : R(M) \to A$, then we can recover $M$ from $A$. Let
  us describe how this is done.
  
  The close almost-bijection gives us $$\alpha_1 = \limsup_{j \to \infty} A_{j+1} -
  A_j,$$ this is the content of \cite{GPPS}[Lemma 2.6]. We then write $A^{(1)} =
  A$, and $R_1 = R(M)$. Assuming that for $2 \le m \le \ell+1$ we have found
  $\alpha_{m-1}$, we write
  $$R_m = R_{m-1} \setminus (\alpha_{m-1} \N \cup
  \alpha_{m-1} \N).$$
Assuming that there is a close almost-bijection $F_{m-1} : R_{m-1} \to
A^{(m-1)}$, there is $K \in \N$ such that for all $k \ge N$, there are at least
two elements of $A^{(m-1)}$ at distance less than $\alpha_1/10$, say, from
$\alpha_{m-1} k$. Construct $A^{(m)}$ by removing from $A^{(m-1)}$ the two
closest such elements (in case of ties choose the largest). 

\cite{GPPS}[Lemma 2.8] states that the existence of a close almost-bijection
$F_{m-1} : R_{m-1} \to A^{(m-1)}$ implies that there is a close almost-bijection
$F_m : R_m \to A^{(m)}$. Applying recursively \cite{GPPS}[Lemma 2.6] gives us 
\begin{equation}
  \alpha_m = \limsup_{j \to \infty} A^{(m)}_{j+1} - A^{(m)}_j,
\end{equation}
which is the necessary ingredient for this recursion to continue.
This is done until $A^{(\ell + 1)}$ and $R_{\ell+1}$  are finite, at which point
$M$ is exhausted and we have recovered $M$ from $A$.
  
  Now, it is not hard to see that the map $F : R(M) \to S(\Xi^{(N)})$
  that maps $R(M)_j$ to $S(\Xi^{(N)})_j$ is a close almost-bijection. Indeed, it follows from the definition of the sequences $\xi^{(m,N)}$ that
\begin{equation}
S(\Xi^{(N)})_j = R(M)_j + \bigo{j^{-1}}
\end{equation}
which implies that $F$ is a close almost-bijection. Our previous analysis tells
us that we can recover $M$ from $S(\Xi^{(N)})$, for any $N\ge 1$.

  \textbf{Step 2: } Suppose without loss of generality that the smallest element of $M$ is $1$.
  Define on positive real numbers the strict partial order $x \prec y$ if there
  is an integer $n \ge 2$ such that $y = nx$, and denote by $x \preceq y$ the
  non-strict version of this partial order, i.e. if $n = 1$ is allowed. For any multiset $U$ of positive
  real numbers, we
  say that $x \in U$ is minimal in $U$ if for all $y \in U$, either $x \preceq y$,
  or $x$ and $y$ are incomparable. Let $I \subset \set{1,\dotsc,\ell}$ be
  defined as
  \begin{equation}
    I = \set{1 \le m \le \ell : \alpha_m \text{ is minimal in } M}.
  \end{equation}
  We claim that there exist $\delta > 0$ and subsets $E_m \subset \N$ of
  infinite cardinality for each $m \in I$ such that for all $j \in E_m$, 
  \begin{equation}
    \label{eq:isolatedintersection}
    [j \alpha_m - \delta, j \alpha_m +\delta] \cap R(M) =
    \underbrace{\set{j\alpha_m,\dotsc,j\alpha_m}}_{2\mu(m) \text{ times}},
  \end{equation}
  where $\mu(m)$ is the multiplicity of $\alpha_m$ in $M$.

  Split $M$ into $M_1 \cup M_2$, where $M_1 \subset \Q$ and $M_2 \subset \R
  \setminus \Q$. Let $Q$ be the smallest common integer multiple of elements in
  $M_1$. Dirichlet's simulateneous approximation theorem states that there is an
  infinite subset $E \subset \N$ such that for all $q \in E$ and $\alpha_m \in
  M_2$ there exists $p_{q,m} \in \N$ such that
  \begin{equation}
    \abs{\frac{Q}{\alpha_m} - \frac{p_{q,m}}{q}} < \frac{1}{q^{1+1/\ell}}
  \end{equation}
  or, equivalently,
  \begin{equation}
    \abs{Qq - p_{q,m} \alpha_m} < \alpha_m q^{-1/\ell}.
  \end{equation}
  This means that for all $q \in E$, there is an integer multiple of $\alpha_m$
  within $q^{-1/\ell}$ of $qQ$. Note that for $\alpha_m \in M_1$, the integer multiple is
  actually exactly $qQ$. In that case we put $p_{q,m} = Q q \alpha_m^{-1}$. Set
  \begin{equation}
    \delta = \frac 1 2 \min
    \set{\abs{\alpha_m - n \alpha_k} : m \in I, \alpha_k \ne \alpha_m, n \in \N},
  \end{equation}
  and observe that $\delta > 0$ from the assumption that $\alpha_m$
  is minimal in $M$ for all $m \in I$. Assume that $\alpha_\ell$ is the largest element of $M$ and for $m \in I$, set
  \begin{equation}
    \label{eq:Em}
    E_m :=\set{p_{q,m} + 1 : q \in E, q^{-1/\ell} < \frac{\delta}{2\alpha_\ell}}.
  \end{equation}
  We claim that for all $j \in E_m$, \eqref{eq:isolatedintersection} holds.
  Indeed, if $\alpha_k \ne \alpha_m$ and $n \in \N$, we have
  \begin{equation}
    \begin{aligned}
    \abs{j \alpha_m - n \alpha_k} &= \abs{(p_{q,m} + 1)\alpha_m - (p_{q,k} + n')
    \alpha_k} \\
    &\ge \abs{\alpha_m - n' \alpha_k} - \abs{p_{q,m} \alpha_m - p_{q,k}
    \alpha_k} \\
    &\ge 2 \delta - (\alpha_m + \alpha_k) q^{-1/\ell} \\
    &> \delta.
  \end{aligned}
  \end{equation}
  It follows that no integer multiple of $\alpha_k \ne \alpha_m$ is within
  distance $\delta$ of $j \alpha_m$, when $j \in E_m$. On the other hand, by
  definition of $R(M)$, and assuming without loss of generality that $\delta <
1$, $j \alpha_m $ is the only integer multiple of $\alpha_m$ in the interval $[j
\alpha_m - \delta, j \alpha_m + \delta]$, and this happens with multiplicity
$2\mu(m)$. 

\textbf{Step 3: } For $m \in I$, we recover the quantities $s_n^{(k)}$ from $S(\Xi^{(N)})$ for any $n \leq N$ and for all $k$ such that $\alpha_k = \alpha_m$.

Let $j \in E_m$ and observe that for any $N \in \N$, the indices in the sequence $S(\Xi^{(N)})$ for the elements in the interval $[j \alpha_m - \delta, j \alpha_m + \delta]$ can be uniquely determined from $R(M)$, which is determined by $S(\Xi^{(N)})$ as seen in the first step of this proof. It follows from \eqref{eq:aexpabstract}, that for all $k$ such that $\alpha_m = \alpha_k$ and $j \in E_m$ large enough, we have
\begin{equation}
  \label{eq:interseq}
  \set{\xi^{(k,N)}_p : p \in \N} \cap [j \alpha_m - \delta, j \alpha_m + \delta] =
    \set{\xi^{(k,N)}_{2j-1},\xi^{(k,N)}_{2j}}.
\end{equation}
For any $N \geq 1$, consider the set
\begin{equation}
X_1^{(m,N)} = \{(x - j\alpha_m) j  : j\in E_m, x \in S(\Xi^{(N)}) \cap [j \alpha_m - \delta, j\alpha_m + \delta]\}.
\end{equation}
From the definition of $E_m$, we have
\begin{equation}
X_1^{(m,N)} = \bigcup_{k: \alpha_k = \alpha_m} \set{(\xi^{(k,N)}_{2j-1} - j\alpha_m)j, (\xi^{(k,N)}_{2j} - j \alpha_m)j}_{j \in E_m}.
\end{equation}
Consider the limit points of $X_1^{(m,N)}$. We claim that those points are exactly the values of $s_1^{(k)}$ for which $\alpha_k = \alpha_m$ . In fact, from the previous equation, $X_1^{(m,N)}$ is a union of sequences and the claim follows from the fact that
\begin{equation}
\lim_{\substack{j \to \infty \\ j \in E_m}} (\xi^{(k,N)}_{2j-1} - j\alpha_m)j = \lim_{\substack{j \to \infty \\ j \in E_m}} (\xi^{(k,N)}_{2j} - j\alpha_m)j = s_{1}^{(k)}.
\end{equation}
Moreover, we can know the number of $k'$ such that $s_1^{(k')} = s_1^{(k)}$, which we denote by $\mathrm{mult}(s_1^{(k)})$. Indeed, by setting
\begin{equation}
\varepsilon = \frac{1}{2} \min\set{\abs{s_1^{(k)} - s_1^{(k')}} : s_1^{(k)} \neq s_1^{(k')}, \alpha_k = \alpha_m},
\end{equation}
we have that $\frac{\mathrm{mult}(s_1^{(k)})}{\mu(m)}$ is given by
\begin{equation}
\lim_{J\rightarrow \infty} \frac{\abs{\{(x - j\alpha_m)j \in X_1^{(m,N)} \cap (s_1^{(k)} - \varepsilon, s_1^{(k)} + \varepsilon) : j \in E_m, j \leq J\}}}{2\abs{\{j \in E_m : j \leq J\}}}.
\end{equation}
Note that from the construction, we cannot directly know which $k$ is associated to each $s_1^{(k)}$, but without loss of generality we can label them in any way we choose since we know their multiplicity. For $k$ with $\alpha_k = \alpha_m$, we construct the sequences
\begin{equation}
\eta_{j}^{(k,1)} = j\alpha_m + s_{1}^{(k)}j^{-1}
\end{equation}
taking into account the multiplicity of $s_1^{(k)}$. We let $\mult(\eta_j^{(k,1)})$ be the number of such sequences identical to $\eta_j^{(k,1)}$. In this case, $\mult(\eta_j^{(k,1)}) = \mult(s_1^{(k)})$. Moreover, $\eta_j^{(k,1)}$ is determined by $S(\Xi^{(N)})$ for any $N \geq 1$.

Suppose now that we know $s_1^{(k)}, \dots, s_T^{(k)}$ from $S(\Xi^{(N)})$ for all $N \geq T$ and $k$ for which $\alpha_k = \alpha_m$, and consider the sequences
\begin{equation}
\eta_{j}^{(k,T)} = j\alpha_m + \sum_{n=1}^T s_n^{(k)}j^{-n}.
\end{equation}
As previously, for $N \geq T+1$, consider the set
\begin{equation}
X_{T+1}^{(k,N)} = \{(x - \eta_j^{(k,T)})j^{T+1} : j\in E_m, x \in S(\Xi^{(N)}) \cap [j \alpha_m - \delta, j\alpha_m + \delta]\}
\end{equation}
which we can rewrite as
\begin{equation}
X_{T+1}^{(k,N)} = \bigcup_{k' : \alpha_{k'} = \alpha_m} \set{(\xi^{(k',N)}_{2j-1} - \eta_j^{(k,T)})j^{T+1}, (\xi^{(k',N)}_{2j} - \eta_j^{(k,T)})j^{T+1}}_{j \in E_m}.
\end{equation}
We claim that the limit points of $X_{T+1}^{(k,N)}$ are precisely the coefficients $s_{T+1}^{(k')}$ such that $\eta_j^{(k',T)} = \eta_j^{(k,T)}$ for all $j \in \N$. This follows from the fact that
\begin{align}
\lim_{\substack{j \to \infty \\ j \in E_m}} (\xi^{(k',N)}_{2j-1} - \eta_{j}^{(k,T)})j^{T+1} &= \lim_{\substack{j \to \infty \\ j \in E_m}} (\xi^{(k',N)}_{2j} - \eta_{j}^{(k,T)})j^{T+1} \\
&=
\begin{cases}
s_{T+1}^{(k')} & \text{if } \eta_j^{(k',T)} = \eta_j^{(k,T)}, \\
\pm\infty & \text{otherwise.}
\end{cases}
\end{align}
We can also deduce the multiplicity of each $s_{T+1}^{(k)}$ in a similar fashion as before. It follows that we can construct the sequences
\begin{equation}
\eta_j^{(k,T+1)} = j\alpha_m + \sum_{n=1}^{T+1} s_n^{(k)} j^{-n}
\end{equation}
and we know the multiplicity of each such sequence. By induction, we can then deduce any coefficient $s_n^{(k)}$ from $S(\Xi^{(N)})$ as long as $N \geq n$.

\textbf{Step 4: } We now turn our attention to $m \not \in I$, and assume that
we have already proved the proposition for all $k$ such that $\alpha_k \prec
\alpha_m$. Defining this time
\begin{equation}
  \delta = \frac 1 2 \min \set{\abs{\alpha_m - n \alpha_k} : \alpha_k \not
    \preceq
  \alpha_m, n \in \N}
\end{equation}
and $E_m$ as in \eqref{eq:Em}, it follows from the same construction as in Step 2
that
\begin{equation}
  [j\alpha_m - \delta, j \alpha_m + \delta] \cap R(M) = \underbrace{\set{j
  \alpha_m,\dotsc,j\alpha_m}}_{\mu \text{ times}},
\end{equation}
where $\mu = 2\sum_{\alpha_k \preceq \alpha_m}
\mu(k)$. We observe that once again, the indices in the sequence $S(\Xi^{(N)})$ of those elements are uniquely determined by $R(M)$ for any $N \geq 1$. For every $k$ such that $\alpha_k \preceq \alpha_m$, write $r(k)$ to be the integer such that $\alpha_m = r(k) \alpha_k$. Defining $X_1^{(m,N)}$ as in step $3$, its limit points are now given by the values of $\frac{s_1^{(k)}}{r(k)}$ for which $\alpha_k \preceq \alpha_m$. From the induction hypothesis, we know those values whenever $r(k) > 1$, so we can disregard them. What is left are the values of $s_1^{(k)}$ for which $\alpha_k = \alpha_m$. Proceeding in a similar manner as in step $3$, but with
\begin{equation}
\eta_j^{(k,T+1)} = j\alpha_m + \sum_{n=1}^{T+1} s_n^{(k)}(r(k)j)^{-n}
\end{equation}
and disregarding the values we already know, we are then able to recover recursively the values of $s_n^{(k)}$ from $S(\Xi^{(N)})$ for any $n \leq N$ for each $k$ with $\alpha_k = \alpha_m$. The set $M$ is finite, hence our inductive procedure necessarily terminates, finishing the proof.
\end{proof}

Theorem \ref{thm:invariance} follows directly from Proposition
\ref{prop:decoupling}.

\begin{proof}[Proof of Theorem \ref{thm:invariance}]
Let $N \in \N$ and for $1 \leq m \leq \ell$, let $\xi^{(m,N)}$ and $\Xi^{(N)}$ be as defined in the statement of Theorem \ref{thm:mc}. Since $\sigma_j \sim S(\Xi^{(\infty)})_j$, there exists a constant $C_N$ such that for all $j \geq 0$,
\begin{equation}\label{eq:sigmaXiN}
\abs{\sigma_j - S(\Xi^{(N)})_j} \leq C_N j^{-N-1}.
\end{equation}
We define a new set of sequences $\tilde{\xi}^{(m,N)}$ by
\begin{equation}\label{eq:tildexi}
\tilde{\xi}_j^{(m,N)} = \xi_j^{(m,N)} + \sigma_{\iota(m,j)} - \xi_j^{(m,N)}
\end{equation}
where $\iota: \set{1, \dots, \ell} \times \N_0 \rightarrow \N_0$ is a bijective map such that 
\begin{equation}
S(\Xi^{(N)})_{\iota(m,j)} = \xi_j^{(m,N)}. 
\end{equation}
By \eqref{eq:sigmaXiN}, we have
\begin{equation}
\tilde{\xi}_j^{(m,N)} = \xi_j^{(m,N)} + O_N(j^{-N-1}).
\end{equation}
Therefore, the sequences $\tilde{\xi}_j^{(m,N)}$ satisfy the hypotheses of Proposition \ref{prop:decoupling} with $M= \set{\frac{2\pi}{\per(\Sigma_1)},\dotsc,\frac{2\pi}{\per(\Sigma_\ell)}}$. Letting $\tilde{\Xi}^{(N)} = \set{\tilde{\xi}^{(1,N)}, \dots, \tilde{\xi}^{(\ell, N)}}$, it follows that $S(\tilde{\Xi}^{(N)})$ determines $M$ and the coefficients $s_n^{(m)}$ for $1 \leq m \leq \ell$ and $n \leq N$. By taking $N$ arbitrarily large, we can recover any coefficient $s_n^{(m)}$. On the other hand, by \eqref{eq:tildexi}, every $\tilde{\xi}_j^{(m,N)}$ is simply an eigenvalue $\sigma_j$, and since $\iota$ is a bijection,
\begin{equation}
S(\tilde{\Xi}^{(N)})_j = \sigma_j
\end{equation}
for all $j \geq 0$. It follows that the spectrum of $\DN_\lambda$ is given by $S(\tilde{\Xi}^{(N)})$ for any $N \in \N$ and therefore determines the number of connected components of the boundary and their respective perimeters (from $M$), as well as all the coefficients $s_n^{(m)}$.

  In particular, from \eqref{eq:particularparam} we have
  \begin{equation}
    s_{-1}^{(m)}(\lambda;\Omega) = - \frac{\lambda \per(\Sigma_m)}{4\pi},
  \end{equation}
  allowing us to recover $\lambda$, and
  \begin{equation}
    s_{-2}^{(m)}(\lambda;\Omega) = \frac{\lambda \per(\Sigma_m)}{8\pi}
    \int_{\Sigma_m} k_g \de s,
  \end{equation}
  allowing us to recover the total geodesic curvature on each boundary component.
\end{proof}

We can now prove Theorem \ref{cor:genusetc} as well.
\begin{proof}[Proof of Theorem \ref{cor:genusetc}]
Since the total geodesic curvature on each boundary component is a spectral invariant, the total integral
  \begin{equation}
    \int_{\Sigma} k_g \de s = \sum_{m=1}^{\ell} \int_{\Sigma_m} k_g \de s
  \end{equation}
is a spectral invariant.
Applying the Gauss-Bonnet theorem, we get
\begin{equation}
\int_{\Sigma} k_g \de s = 2\pi(2 - 2\gamma - \ell) - \int_{\Omega} K_g \de A_g
\end{equation}
where $\gamma$ is the genus of $\Omega$. Since the number of boundary components $\ell$ is a spectral invariant, we can deduce that the quantity
\begin{equation}
4\pi \gamma + \int_{\Omega}K_g \de A_g
\end{equation}
is also a spectral invariant of the constant potential parametric Steklov problem. 
\end{proof}
\begin{rem}
It is
impossible to completely decouple the genus and the average of the Gaussian
curvature as spectral invariants from the eigenvalue
asymptotic expansion since the addition of a handle far from the boundary
changes the genus of $\Omega$ but leaves the symbol of the Dirichlet-to-Neumann
operator unchanged. However, a priori information on $\Omega$, such as being a
domain of a specific space form of constant Gaussian curvature can yield additional information, as in
Corollaries \ref{cor:sphere} and \ref{cor:flat}.
\end{rem}

\bibliographystyle{plain}

\nocite{*}

\bibliography{reference}

\begin{thebibliography}{10}

\bibitem{agranovichrus}
Mikhail~S. Agranovich.
\newblock Elliptic pseudodifferential operators on a closed curve.
\newblock {\em Trudy Moskovskogo Matematicheskogo Obshchestva}, 47:22--67,
  1984.

\bibitem{ADGHRS}
Teresa Arias-Marco, Emily~B. Dryden, Carolyn~S. Gordon, Asma Hassannezhad,
  Allie Ray, and Elizabeth Stanhope.
\newblock Spectral geometry of the {S}teklov problem on orbifolds.
\newblock {\em International Mathematics Research Notices}, 2019(1):90--139,
  2019.

\bibitem{CCMM}
Fioralba Cakoni, David Colton, Shixu Meng, and Peter Monk.
\newblock Stekloff eigenvalues in inverse scattering.
\newblock {\em SIAM Journal on Applied Mathematics}, 76(4):1737--1763, 2016.

\bibitem{DG}
Johannes~J. Duistermaat and Victor~W. Guillemin.
\newblock The spectrum of positive elliptic operators and periodic
  bicharacteristics.
\newblock {\em Inventiones Mathematicae}, 29:39--79, 1975.

\bibitem{edwards}
Julian Edward.
\newblock An inverse spectral result for the {N}eumann operator on planar
  domains.
\newblock {\em Journal of Functional Analysis}, 111(2):312--322, 1993.

\bibitem{edward2}
Julian Edward.
\newblock Pre-compactness of isospectral sets for the {N}eumann operator on
  planar domains.
\newblock {\em Communications in Partial Differential Equations},
  18(7-8):1249--1270, 1993.

\bibitem{egorovschultze}
Yuri Egorov and Bert-Wolfgang Schulze.
\newblock {\em Pseudo-differential operators, singularities, applications},
  volume~93.
\newblock Birkh{\"a}user, 1997.

\bibitem{GPPS}
Alexandre Girouard, Leonid Parnovski, Iosif Polterovich, and David~A. Sher.
\newblock The {S}teklov spectrum of surfaces: asymptotics and invariants.
\newblock In {\em Mathematical Proceedings of the Cambridge Philosophical
  Society}, volume 157, pages 379--389. Cambridge University Press, 2014.

\bibitem{girouardpolterovich}
Alexandre Girouard and Iosif Polterovich.
\newblock Spectral geometry of the {S}teklov problem (survey article).
\newblock {\em Journal of Spectral Theory}, 7(2):321--360, 2017.

\bibitem{GWW}
Carolyn Gordon, David~L. Webb, and Scott Wolpert.
\newblock One cannot hear the shape of a drum.
\newblock {\em Bulletin of the American Mathematical Society}, 27(1):134--138,
  1992.

\bibitem{hormanderiv}
Lars H{\"o}rmander.
\newblock {\em The analysis of partial differential operators, vol. IV}.
\newblock Springer, 1984.

\bibitem{howard}
Ralph Howard.
\newblock Blaschke's rolling theorem for manifolds with boundary.
\newblock {\em Manuscripta Mathematica}, 99(4):471--483, 1999.

\bibitem{JolSharaf}
Alexandre Jollivet and Vladimir Sharafutdinov.
\newblock On an inverse problem for the {S}teklov spectrum of a {R}iemannian
  surface.
\newblock {\em Contemp. Math}, 615:165--191, 2014.

\bibitem{JolSharaf2018}
Alexandre Jollivet and Vladimir Sharafutdinov.
\newblock Steklov zeta-invariants and a compactness theorem for isospectral
  families of planar domains.
\newblock {\em Journal of Functional Analysis}, 275(7):1712--1755, 2018.

\bibitem{kac}
Mark Kac.
\newblock Can one hear the shape of a drum?
\newblock {\em The American Mathematical Monthly}, 73(4P2):1--23, 1966.

\bibitem{leeuhlmann}
John~M. Lee and Gunther Uhlmann.
\newblock Determining anisotropic real-analytic conductivities by boundary
  measurements.
\newblock {\em Communications on Pure and Applied Mathematics},
  42(8):1097--1112, 1989.

\bibitem{liu}
Genqian Liu.
\newblock Asymptotic expansion of the trace of the heat kernel associated to
  the {D}irichlet-to-{N}eumann operator.
\newblock {\em Journal of Differential Equations}, 259(7):2499--2545, 2015.

\bibitem{PS}
Iosif Polterovich and David~A. Sher.
\newblock Heat invariants of the {S}teklov problem.
\newblock {\em The Journal of Geometric Analysis}, 25(2):924--950, 2015.

\bibitem{rozenblumalmostsim}
Grigori~V. Rozenblum.
\newblock Almost-similarity of operators and spectral asymptotics of
  pseudodifferential operators on a circle.
\newblock {\em Trans. Moscow. Math. Soc.}, 2:59--84, 1979.
\newblock Translated from Russian.

\bibitem{rozenblumstek}
Grigori~V. Rozenblum.
\newblock On the asymptotics of the eigenvalues of certain two-dimensional
  spectral problems.
\newblock {\em Sel. Math. Sov}, 5:233--244, 1986.

\bibitem{taylorpartial}
Michael Taylor.
\newblock {\em Partial differential equations II: Qualitative studies of linear
  equations}, volume 116.
\newblock Springer, 2013.

\bibitem{WW}
Weiwei Wang and Zuoqin Wang.
\newblock On the relative heat invariants of the {D}irichlet-to-{N}eumann
  operators associated with {S}chr{\"o}dinger operators.
\newblock {\em Journal of Pseudo-Differential Operators and Applications},
  10(4):805--836, 2019.

\end{thebibliography}

\end{document}